\documentclass [reqno]{amsart}

\usepackage{amsmath,amsfonts,amssymb,amscd,verbatim,delarray,fullpage}

\usepackage{stmaryrd}
\usepackage{amsthm}
\usepackage{url}
\usepackage[polutonikogreek,english]{babel}
\usepackage{graphicx}
\usepackage{float}

\numberwithin{equation}{section}

\DeclareMathOperator{\prob}{Prob}

\DeclareMathOperator{\cyc}{c}

\DeclareMathOperator{\Li}{Li}

\chardef\bslash=`\\ 

\begin{document}


\newtheorem{Theorem}{Theorem}[section]

\newtheorem{cor}[Theorem]{Corollary}

\newtheorem{Conjecture}[Theorem]{Conjecture}

\newtheorem{Lemma}[Theorem]{Lemma}
\newtheorem{lemma}[Theorem]{Lemma}
\newtheorem{property}[Theorem]{Property}
\newtheorem{proposition}[Theorem]{Proposition}
\newtheorem{ax}[Theorem]{Axiom}
\newtheorem{claim}[Theorem]{Claim}

\newtheorem{nTheorem}{Surjectivity Theorem}

\theoremstyle{definition}
\newtheorem{Definition}[Theorem]{Definition}
\newtheorem{problem}[Theorem]{Problem}
\newtheorem{question}[Theorem]{Question}
\newtheorem{Example}[Theorem]{Example}

\newtheorem{remark}[Theorem]{Remark}
\newtheorem{diagram}{Diagram}
\newtheorem{Remark}[Theorem]{Remark}
\newcommand{\diagref}[1]{diagram~\ref{#1}}
\newcommand{\thmref}[1]{Theorem~\ref{#1}}
\newcommand{\secref}[1]{Section~\ref{#1}}
\newcommand{\subsecref}[1]{Subsection~\ref{#1}}
\newcommand{\lemref}[1]{Lemma~\ref{#1}}
\newcommand{\corref}[1]{Corollary~\ref{#1}}
\newcommand{\exampref}[1]{Example~\ref{#1}}
\newcommand{\remarkref}[1]{Remark~\ref{#1}}
\newcommand{\corlref}[1]{Corollary~\ref{#1}}
\newcommand{\claimref}[1]{Claim~\ref{#1}}
\newcommand{\defnref}[1]{Definition~\ref{#1}}
\newcommand{\propref}[1]{Proposition~\ref{#1}}
\newcommand{\prref}[1]{Property~\ref{#1}}
\newcommand{\itemref}[1]{(\ref{#1})}


\newcommand{\CE}{\mathcal{E}}
\newcommand{\CG}{\mathcal{G}}\newcommand{\CV}{\mathcal{V}}
\newcommand{\CL}{\mathcal{L}}
\newcommand{\CM}{\mathcal{M}}
\newcommand{\A}{\mathcal{A}}
\newcommand{\CO}{\mathcal{O}}
\newcommand{\B}{\mathcal{B}}
\newcommand{\CS}{\mathcal{S}}
\newcommand{\CX}{\mathcal{X}}
\newcommand{\CY}{\mathcal{Y}}
\newcommand{\CT}{\mathcal{T}}
\newcommand{\CW}{\mathcal{W}}
\newcommand{\CJ}{\mathcal{J}}

\newcommand{\st}{\sigma}
\renewcommand{\k}{\varkappa}
\newcommand{\Frac}{\mbox{Frac}}
\newcommand{\XC}{\mathcal{X}}
\newcommand{\wt}{\widetilde}
\newcommand{\wh}{\widehat}
\newcommand{\mk}{\medskip}
\renewcommand{\sectionmark}[1]{}
\renewcommand{\Im}{\operatorname{Im}}
\renewcommand{\Re}{\operatorname{Re}}
\newcommand{\la}{\langle}
\newcommand{\ra}{\rangle}
\newcommand{\LND}{\mbox{LND}}
\newcommand{\Pic}{\mbox{Pic}}
\newcommand{\lnd}{\mbox{lnd}}
\newcommand{\GLND}{\mbox{GLND}}\newcommand{\glnd}{\mbox{glnd}}
\newcommand{\Der}{\mbox{DER}}\newcommand{\DER}{\mbox{DER}}
\renewcommand{\th}{\theta}
\newcommand{\ve}{\varepsilon}
\newcommand{\1}{^{-1}}
\newcommand{\iy}{\infty}
\newcommand{\iintl}{\iint\limits}
\newcommand{\capl}{\operatornamewithlimits{\bigcap}\limits}
\newcommand{\cupl}{\operatornamewithlimits{\bigcup}\limits}
\newcommand{\suml}{\sum\limits}
\newcommand{\ord}{\operatorname{ord}}
\newcommand{\gal}{\operatorname{Gal}}
\newcommand{\bk}{\bigskip}
\newcommand{\fc}{\frac}
\newcommand{\g}{\gamma}
\newcommand{\be}{\beta}
\newcommand{\dl}{\delta}
\newcommand{\Dl}{\Delta}
\newcommand{\lm}{\lambda}
\newcommand{\Lm}{\Lambda}
\newcommand{\om}{\omega}
\newcommand{\ov}{\overline}
\newcommand{\vp}{\varphi}
\newcommand{\kap}{\varkappa}

\newcommand{\Vp}{\Phi}
\newcommand{\Varphi}{\Phi}
\newcommand{\BC}{\mathbb{C}}
\newcommand{\C}{\mathbb{C}}\newcommand{\BP}{\mathbb{P}}
\newcommand{\BQ}{\mathbb {Q}}
\newcommand{\BM}{\mathbb{M}}
\newcommand{\BR}{\mathbb{R}}\newcommand{\BN}{\mathbb{N}}
\newcommand{\BZ}{\mathbb{Z}}\newcommand{\BF}{\mathbb{F}}
\newcommand{\BA}{\mathbb {A}}
\renewcommand{\Im}{\operatorname{Im}}
\newcommand{\idd}{\operatorname{id}}
\newcommand{\ep}{\epsilon}
\newcommand{\tp}{\tilde\partial}
\newcommand{\doe}{\overset{\text{def}}{=}}
\newcommand{\supp} {\operatorname{supp}}
\newcommand{\loc} {\operatorname{loc}}
\newcommand{\de}{\partial}
\newcommand{\z}{\zeta}
\renewcommand{\a}{\alpha}
\newcommand{\G}{\Gamma}
\newcommand{\der}{\mbox{DER}}

\newcommand{\Spec}{\operatorname{Spec}}
\newcommand{\Sym}{\operatorname{Sym}}
\newcommand{\aut}{\operatorname{Aut}}
\newcommand{\End}{\operatorname{End}}

\newcommand{\Idd}{\operatorname{Id}}

\newcommand{\tG}{\widetilde G}

\newcommand{\FX}{\mathfrac {X}}
\newcommand{\FV}{\mathfrac {V}}
\newcommand{\SX}{\mathcal {X}}
\newcommand{\SV}{\mathcal {V}}
\newcommand{\SO}{\mathcal {O}}
\newcommand{\SD}{\mathcal {D}}
\newcommand{\Sr}{\rho}
\newcommand{\SR}{\mathcal {R}}
\newcommand{\cl}{\mathcal{C}}
\newcommand{\ok}{\mathcal{O}_K}
\newcommand{\ab}{\mathfrak{Ab}}

\setcounter{equation}{0} \setcounter{section}{0}

\newcommand{\ds}{\displaystyle}
\newcommand{\gl}{\lambda}
\newcommand{\gL}{\Lambda}
\newcommand{\gge}{\epsilon}
\newcommand{\gG}{\Gamma}
\newcommand{\ga}{\alpha}
\newcommand{\gb}{\beta}
\newcommand{\gd}{\delta}
\newcommand{\gD}{\Delta}
\newcommand{\gs}{\sigma}
\newcommand{\mbq}{\mathbb{Q}}
\newcommand{\mbr}{\mathbb{R}}
\newcommand{\mbz}{\mathbb{Z}}
\newcommand{\mbc}{\mathbb{C}}
\newcommand{\mbn}{\mathbb{N}}
\newcommand{\mbp}{\mathbb{P}}
\newcommand{\mbf}{\mathbb{F}}
\newcommand{\mbe}{\mathbb{E}}
\newcommand{\mby}{\mathbb{Y}}
\newcommand{\lcm}{\text{lcm}\,}
\newcommand{\mf}[1]{\mathfrak{#1}}
\newcommand{\ol}[1]{\overline{#1}}
\newcommand{\mc}[1]{\mathcal{#1}}
\newcommand{\nequiv}{\equiv\hspace{-.13in}/\;}

\newcommand{\hd}{h(-D)}
\newcommand{\Hd}{H(-D)}
\newcommand{\Q}{{\mathbb Q}} 
\newcommand{\Z}{{\mathbb Z}} 
\newcommand{\fixme}[1]{}

\newcommand{\E}[1]{\mathbb E\left(#1\right)}
\newcommand{\inv}{^{-1}}
\newcommand{\defeq}{\mathrel{\mathop:}=} 
\newcommand{\eqdef}{=\mathrel{\mathop:}} 
\newcommand{\abs}[1]{\left|#1\right|}
\newcommand{\dd}[1]{\mathop{d#1}}

\title{Missing class groups and class number statistics for imaginary quadratic fields}
\author{S. Holmin, N. Jones, P. Kurlberg, C. McLeman and K. Petersen}
\date{October 4, 2015}

\begin{abstract}
   The number $\mc{F}(h)$ of imaginary quadratic fields with class number $h$ is of classical interest:  Gauss' class number problem asks for a determination of those fields counted by $\mc{F}(h)$.  
The unconditional computation of $\mc{F}(h)$ for $h \leq 100$ was completed by M. Watkins, using ideas of Goldfeld and Gross-Zagier; Soundararajan has more recently made conjectures about the order of magnitude of $\mc{F}(h)$ as $h \rightarrow \infty$ and determined its average order.
  In the present paper, we refine Soundararajan's conjecture to a conjectural asymptotic formula and also
  consider the subtler problem of determining the number $\mc{F}(G)$ of imaginary
  quadratic fields with class group isomorphic to a given finite abelian group $G$.  
  Using Watkins' tables, one can show that some abelian groups do {\em
    not} occur as the class group of any imaginary quadratic field
  (for instance $(\mbz/3\mbz)^3$ does not).  This observation is
  explained in part by the Cohen-Lenstra heuristics, which have
  often been used to study the distribution of the \emph{$p$-part}
  of an imaginary quadratic class group.  We combine
  heuristics of Cohen-Lenstra together with our refinement of Soundararajan's conjecture to make precise predictions about the
  asymptotic nature of the \emph{entire} imaginary quadratic
  class group, in particular addressing the above-mentioned phenomenon of ``missing''
  class groups, for the case of $p$-groups as $p$ tends to infinity.  Furthermore, conditionally on the Generalized Riemann Hypothesis, we extend Watkins' data, tabulating $\mc{F}(h)$ for odd $h \leq 10^6$ and $\mc{F}(G)$ for $G$ a $p$-group of odd order with $|G| \leq 10^6$.  The numerical evidence matches quite well with our conjectures.
\end{abstract}

\maketitle

\section{Introduction} \label{introduction}

Given a fundamental discriminant $d < 0$, let $H(d)$ denote the ideal
class group of the imaginary quadratic field $\Q(\sqrt{d})$, and let
$h(d) := |H(d)|$ denote the class number.  A basic
question is: 
\begin{question} \label{fundamentalquestion}
Which finite abelian groups $G$ occur as $H(d)$ for some negative fundamental
discriminant $d$? 
\end{question}
Equivalently, which finite abelian groups $G$ do \emph{not}
occur as $H(d)$?  
The case where $G \simeq (\mbz/2\mbz)^r$ has classical connections via
genus theory to Euler's ``idoneal numbers,'' and it follows from work
of Chowla \cite{chowla} that for every $r \gg 1$, the group
$(\mbz/2\mbz)^r$ does not occur as the class group of any imaginary
quadratic field.  Later work of various authors
(\cite{boydkisilevsky}, \cite{weinberger}, \cite{heathbrown}) has
shown that $(\Z/n\Z)^r$ does not occur as an imaginary quadratic class
group for $r \gg 1$ and $2 \leq n \leq 6$ (in fact, Heath-Brown showed
that groups with exponent $2^{a}$ or $3 \cdot 2^{a}$ occur only
finitely many times.)  Moreover, $(\Z/n\Z)^r$ does not occur for $n>6$
and $r \gg_{n} 1$ {\em assuming} the Generalized Riemann Hypothesis
(cf. \cite{boydkisilevsky,weinberger}); in fact they show that the
exponent of $H(d)$ tends to infinity as $d \to -\infty$.

  Due to the possible existence of Siegel zeroes,
  the unconditional results mentioned above are ineffective.  To
  find explicit examples of missing class groups, one can undertake a
  brute-force search using tables of M. Watkins \cite{watkins}, who used the ideas of
  Goldfeld and Gross-Zagier to give an unconditional resolution of
  Gauss' class number problem for class numbers $h \leq 100$.  Such a
  search reveals that none of the groups 
  \[
  \left( \frac{\mbz}{3\mbz} \right)^3, \quad
  \frac{\mbz}{9\mbz} \times  \left( \frac{\mbz}{3\mbz} \right)^2, \quad  \left( \frac{\mbz}{3\mbz} \right)^4
  \]
  occur as the class group of an imaginary quadratic field.

  It is also natural to ask how common the groups that do
  occur are:
\begin{question} 
\label{refinedfundamentalquestion} 
Given a finite
  abelian group $G$, for how many fundamental discriminants $d<0$ 
  is $H(d) \simeq G$?  
\end{question}
In order to address this question, we are led to investigate a
closely related issue:
\begin{question}
\label{classnumbercountquestion} 
Given an integer $h>0$, for how many
  fundamental discriminants   $d<0$ is $|H(d)| = h$?
\end{question}

Questions \ref{fundamentalquestion},
\ref{refinedfundamentalquestion}, and \ref{classnumbercountquestion}
appear to be beyond the realm of what one can provably answer in full with current technology.  In this paper, we combine the heuristics of
Cohen-Lenstra with results on the distribution of special values of
Dirichlet $L$-functions to give a conjectural asymptotic answer to
Question \ref{classnumbercountquestion}, for $h$ odd.  (For this we
only use the Cohen-Lenstra heuristic to predict divisibility
properties of class numbers.)  Further, using this conjectured
asymptotic answer, we use the Cohen-Lenstra heuristic to predict
the $p$-group decomposition of $H(d)$ and obtain a conjectured
asymptotic answer to Question \ref{refinedfundamentalquestion} in the
case where $G$ is a $p$-group for an odd prime $p$.  
(We believe that similar results hold for composite class number, 
though here one must be careful in how limits are
taken; for instance with some groups of order $p_1^{n_1} p_2^{n_2}$, $p_1$ fixed and $p_2$ tending to infinity is very different from $p_1$ and $p_2$ both tending to infinity.)
In particular, regarding Question
\ref{fundamentalquestion}, we establish a precise condition on the
\emph{shape} of an abelian $p$-group which governs whether or not it
should occur as an imaginary quadratic class group for infinitely many
primes $p$.  For instance, our conjecture predicts that the group
\[
\left( \frac{\mbz}{p\mbz} \right)^3
\]
should appear as a class group for only finitely many primes $p$ (in
fact, quite likely for no primes $p$ at all;
cf. Conjecture \ref{NoElemAb} in Section~\ref{sec:numer-mathc-fg_l}), whereas the two groups
\[
\frac{\mbz}{p^3\mbz}, \quad \frac{\mbz}{p^2\mbz} \times \frac{\mbz}{p\mbz}
\]
should occur as a class group for infinitely many primes $p$.

Given a positive integer $h$ we set 
\begin{equation} \label{defofFh}
\mc{F}(h) := | \{ \text{fundamental discriminants $d < 0$ } : \; h(d) = h \} |.
\end{equation}
Thus for instance $\mc{F}(1) = 9$, which is the statement of the
Baker-Stark-Heegner theorem on Gauss' class number 1 problem for
imaginary quadratic fields.  Given a fixed finite abelian group $G$,
we consider the refined counting function 
\begin{equation*} 
\mc{F}(G) := | \{ \text{fundamental discriminants $d < 0$ } : \; H(d) \simeq G \} |,
\end{equation*}
so that $\mc{F}(h) = \sum_{|G| = h} \mc{F}(G)$, where the sum runs
over isomorphism classes of finite abelian groups of order $h$.  The
Cohen-Lenstra heuristics suggest that, for
any finite abelian group $G$ of odd order $h$, the expected number of
imaginary quadratic fields with class group $G$ is given by
\begin{equation} \label{factorizationofmcFofG}
\mc{F}(G) \approx P(G) \cdot \mc{F}(h),
\end{equation}
where 
\begin{equation} \label{defofPG}
P(G) :=  
\left( \frac{1}{| \aut(G)|} \right) \big/ 
\left( \sum_{\text{abel. groups }G'  \atop\text{s.t. } |G'| = |G|} \frac{1}{| \aut(G') |} \right).
\end{equation}
The first factor $P(G)$ may
be evaluated explicitly, whereas the second factor $\mc{F}(h)$ is more
delicate.   K. Soundararajan has conjectured (see \cite[p. 2]{sound})
that 
\begin{equation} \label{soundconjectureodd}
\mc{F}(h) \asymp \frac{h}{\log h} \quad\quad\quad\quad \left( h \text{ odd} \right).
\end{equation}
We refine Soundararajan's heuristic, sharpening \eqref{soundconjectureodd} to a conjectural asymptotic formula, which involves certain constants associated to a random Euler product.  Let $\mby = \{ \mathbb{Y}(p) : p \text{ prime}\}$ denote a collection of independent identically distributed random variables satisfying
\begin{equation*} 
\mathbb{Y}(p) :=
\begin{cases}
1 & \text{ with probability } 1/2 \\
-1 & \text{ with probability } 1/2
\end{cases}
\end{equation*}
and let
\[
L(1,\mby) := \prod_{p} \left( 1 - \frac{\mby(p)}{p} \right)^{-1}
\]
denote the corresponding random Euler product, which converges with probability one.  Define the constant
\begin{equation} \label{defofmfC}
\begin{split}
\mf{C} &:= 15 \prod_{\ell = 3 \atop{\ell \text{ prime}}}^\infty \prod_{i = 2}^\infty \left( 1 - \frac{1}{\ell^i} \right) \approx 11.317, \\
\end{split}
\end{equation}
as well as the factor (defined
for odd $h$)
\[
\mf{c}(h) := \prod_{p^n \parallel h} \prod_{i = 1}^n \left( 1 - \frac{1}{p^i} \right)^{-1}.
\]
\begin{Conjecture} \label{refinedsoundconjecture}
We have
\begin{equation} \label{verygoodprediction}
\mc{F}(h) \sim \frac{\mf{C}}{15} \cdot \mf{c}(h) \cdot h \cdot \mbe
\left( \frac{1}{L(1,\mby)^2 \log (\pi h / L(1,\mby))} \right) 
\sim
   \mf{C} \cdot \mf{c}(h)
    \cdot
  \frac{h}{\log(\pi h)}
\end{equation}
as $h \longrightarrow \infty$ through odd values. (Here $\mbe$ denotes expected value.)
\end{Conjecture}

Conjecture \ref{refinedsoundconjecture} is developed from the
Cohen-Lenstra heuristics together with large-scale distributional
considerations of the special value $L(1,\chi_d)$.  The former can be
viewed as a product over non-archimedean primes; the latter as an
archimedean factor --- in a sense our prediction is a ``global'' (or
adelic) generalization of the Cohen-Lenstra heuristic,  somewhat
similar to the Siegel mass formula.

More precisely, motivated by the Cohen-Lenstra heuristic we introduce a
correction factor that considers divisibility of $h$ by a random odd
positive integer (for instance a random class number is divisible by
$3$ with conjectural probability
\[
1 - \prod_{i = 1}^\infty \left( 1 - \frac{1}{3^i} \right) \approx 43 \%,
\]
and this suggests a correction factor of
$\left(1 - \prod_{i = 1}^\infty \left( 1 - \frac{1}{3^i} \right)
\right) / (1/3)$
whenever $3$ divides $h$).  We remark
that the Cohen-Lenstra heuristics have often been applied to give a
probabilistic model governing the \emph{$p$-part} of a class group, for a
\emph{fixed} prime $p$ (see for instance \cite[Section
9]{cohenlenstra}).  By contrast, the precise asymptotic predicted by
Conjecture \ref{refinedsoundconjecture} involves applying these
considerations for \emph{all} primes $p$ (including the archimedean prime).

The relevant information about the distribution of $L(1,\chi_d)$ is implicit in the following theorem, which gives the analogue of \cite[Theorem 1]{sound} averaged over odd values of $h$.
\begin{Theorem} \label{tweakedsoundprop}
Assume the Generalized Riemann Hypothesis.  Then for any $\ve > 0$, we
have
\[
\sum_{h \leq H \atop h \text{ odd}} \mc{F}(h) = \frac{15}{4} \cdot \frac{H^2}{\log H} + O\left( H^2 ( \log H)^{-3/2+\ve} \right),
\]
as $H \longrightarrow \infty$.
\end{Theorem}

\begin{remark}  
In fact, our analysis (cf. Section~\ref{heuristics}) yields the more
accurate 
approximation
\begin{equation} \label{realpred}
\begin{split}
\mc{F}(h) &\sim \frac{\mf{C}}{15} \cdot \mf{c}(h) \cdot h \cdot \mbe
\left( \frac{1}{L(1,\mby)^2 \log (\pi h / L(1,\mby))} \right) \\
&=
\mf{C} \cdot \mf{c}(h) \cdot \frac{h}{\log (\pi h)} \cdot
    \left(
      1
      +\frac{c_1}{\log(\pi h)}
      +\frac{c_2}{\log^2(\pi h)}
      +\frac{c_3}{\log^3(\pi h)}
      +o\left( \frac{1}{\log^3(\pi h)}
      \right)
    \right),
\end{split}
\end{equation}
where
\begin{equation} 
\begin{split}
c_1 &\defeq \frac{\pi^2}{15}\E{\frac{\log L(1,\mathbb Y)}{L(1,\mathbb Y)^2}} \approx -0.578, \\ 
c_2 &\defeq \frac{\pi^2}{15}\E{\frac{\log^2 L(1,\mathbb
    Y)}{L(1,\mathbb Y)^2}} \approx 0.604, \\
c_3 &\defeq \frac1{c_0}
\E{\frac{\log^3 L(1,\mathbb Y)}{L(1,\mathbb Y)^2}} \approx -0.526.
\end{split}
\end{equation}
Without this higher order expansion we have a relative error of
size $O(1/\log h)$; since we only have data for odd $h \leq 10^{6}$,
the  higher order expansion is essential to get a convincing
fit to the observed data.
\end{remark}

With the aid of a supercomputer and assuming GRH, we have
computed $\mathcal F(h)$ and $\mathcal F(G)$ for all odd $h<10^6$ and
all $p$-groups $G$ of odd size at most $10^6$. (For more details, see
Section \ref{numerics}.)  This provides us with numerical evidence in
support of Conjecture \ref{refinedsoundconjecture}.  Below we give
some samples\footnote{The complete list of computed values of
  $\mathcal F(h)$ is given in \cite{data-Fh}.}  of computed values
$\mathcal F(h)$ (conditional on the GRH) compared to the values
predicted by Conjecture \ref{refinedsoundconjecture}, rounded to the
nearest integer. We also list the relative error
$(\mathcal F(h)-\operatorname{pred}(h))/\operatorname{pred}(h)$ given
as a percentage, where
\begin{gather}
  \operatorname{pred}(h) \defeq
  \mf{C} \cdot \mf{c}(h) \cdot \frac{h}{\log (\pi h)} \cdot
    \left(
      1
      +\frac{c_1}{\log(\pi h)}
      +\frac{c_2}{\log^2(\pi h)}
      +\frac{c_3}{\log^3(\pi h)}
    \right).
\end{gather}
\begin{equation*}
\begin{array}{|c|rrrrrrrr|}
\hline
       h         &   10001 &   10003 &   10005 &   10007 &   10009 &   10011 &   10013 &   10015 \\
\hline
 \mathcal F(h)   &   10641 &   12154 &   20661 &   10536 &   10329 &   15966 &   12221 &   12975 \\
\operatorname{pred}(h) &   10598 &   12116 &   21074 &   10383 &   10385 &   16144 &   12038 &   12993 \\
\text{Relative error} & +0.41\% & +0.31\% & -1.96\% & +1.48\% & -0.54\% & -1.10\% & +1.52\% & -0.14\% \\
\hline
\hline
       h         &  100001 &  100003 &  100005 &  100007 &  100009 &  100011 &  100013 &  100015 \\
\hline
 \mathcal F(h)   &   94623 &   85792 &  164289 &   86770 &  111948 &  142512 &   87138 &  108993 \\
\operatorname{pred}(h) &   94213 &   85641 &  164806 &   86620 &  111210 &  142989 &   86577 &  108820 \\
\text{Relative error} & +0.43\% & +0.18\% & -0.31\% & +0.17\% & +0.66\% & -0.33\% & +0.65\% & +0.16\% \\
\hline
\hline
       h         &  999985 &  999987 &  999989 &  999991 &  999993 &  999995 &  999997 &  999999 \\
\hline
 \mathcal F(h)   & 1064529 & 1095135 &  771805 &  791007 & 1093645 &  914482 &  733397 & 1815672 \\
\operatorname{pred}(h) & 1063376 & 1098842 &  769673 &  788871 & 1093732 &  911447 &  730673 & 1825811 \\
\text{Relative error} & +0.11\% & -0.34\% & +0.28\% & +0.27\% & -0.01\% & +0.33\% & +0.37\% & -0.56\% \\
\hline
\end{array}
\end{equation*}
For large $h$ the prediction seems fairly good as the relative
error very often
 is smaller than $1\%$. 
 To gain further insight, we study the fluctuations in the difference
 between the observed data and the predictions, normalized by dividing
 by the square root of the prediction (it is perhaps not a priori
 obvious, but with this normalization the resulting standard deviation
 is close to one in many circumstances).  More precisely, we make a
 histogram of the values of
$$
r(h) := \frac{\mc F(h)-\operatorname{pred}(h) }{\sqrt{\operatorname{pred}(h)}}
$$
for various subsets of the (odd) integers. 
%
For notational convenience, we shall let $\mu$ and $\sigma$ denote the
mean and standard deviation, respectively, of the observed data in
each plot.  

\newcommand{\classfigzero}{
\begin{figure}[H]
  \centering
\includegraphics[width=.6\textwidth]{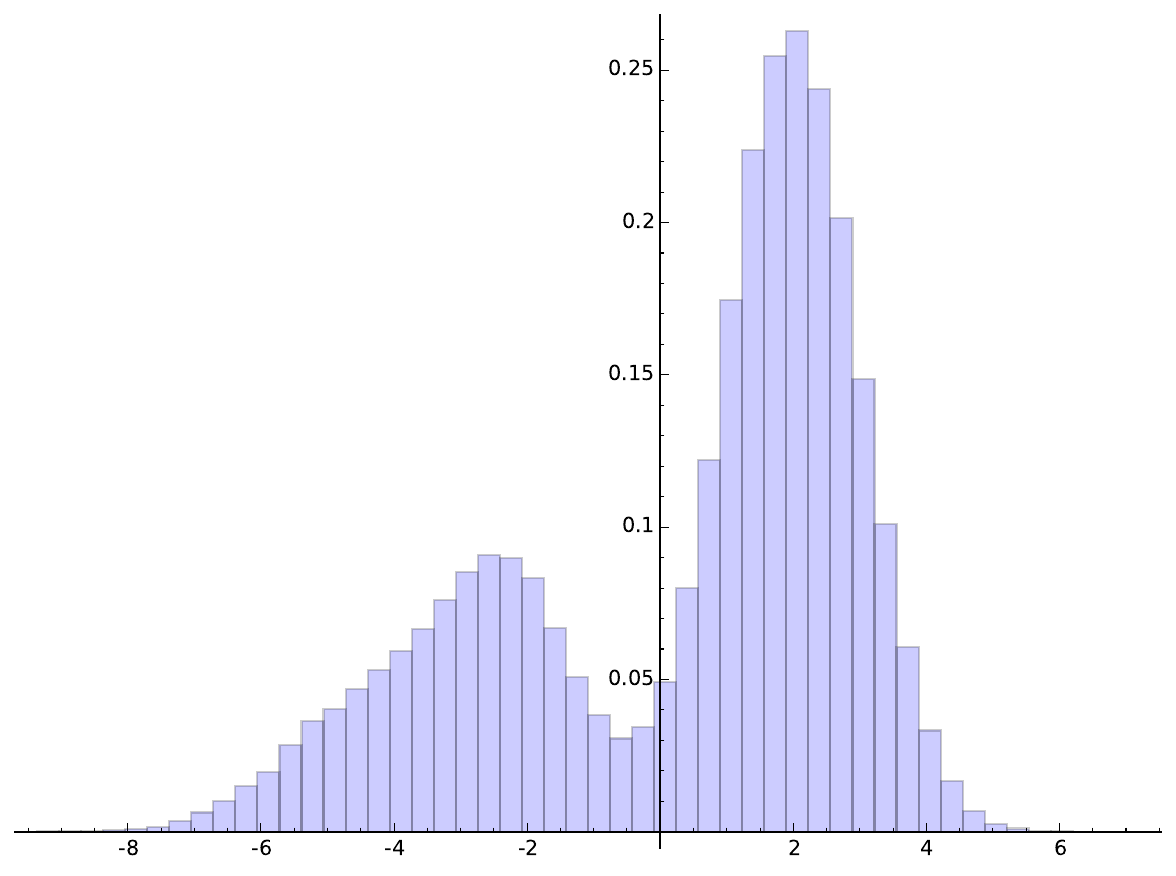}
\caption{Histogram for $r(h)$, as $h$ ranges over
  odd integers in $[500000,1000000]$. $(\mu,\sigma)=(0.291561,2.685280)$.}
\end{figure}}
\newcommand{\classfigtwo}{
\begin{figure}[H]
  \centering
\includegraphics[width=.6\textwidth]{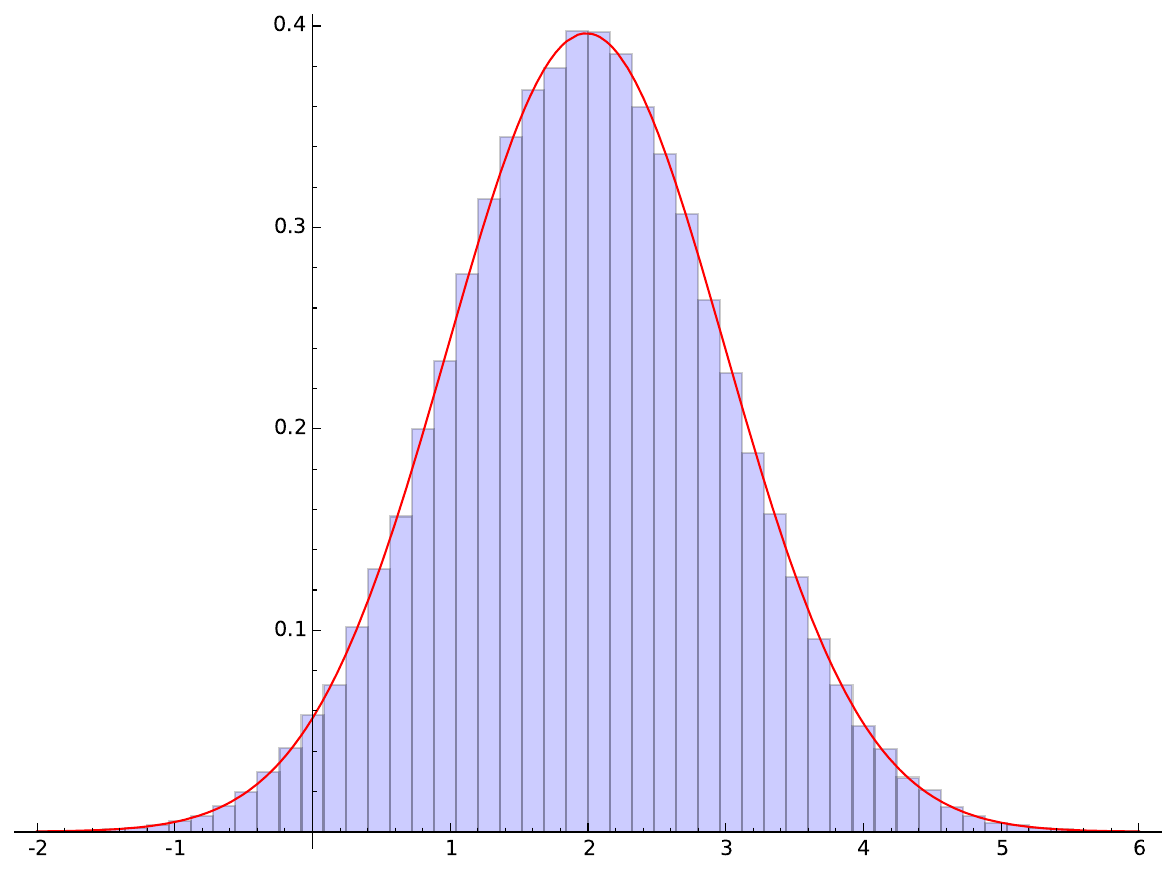}
\caption{Histogram for $r(h)$, as $h \not \equiv 0 \mod 3$ ranges over
  odd integers in $[500000,1000000]$. $(\mu,\sigma)=(1.987995,1.006428)$.}
\end{figure}}
\newcommand{\classfigone}{
\begin{figure}[H]
  \centering
\includegraphics[width=.6\textwidth]{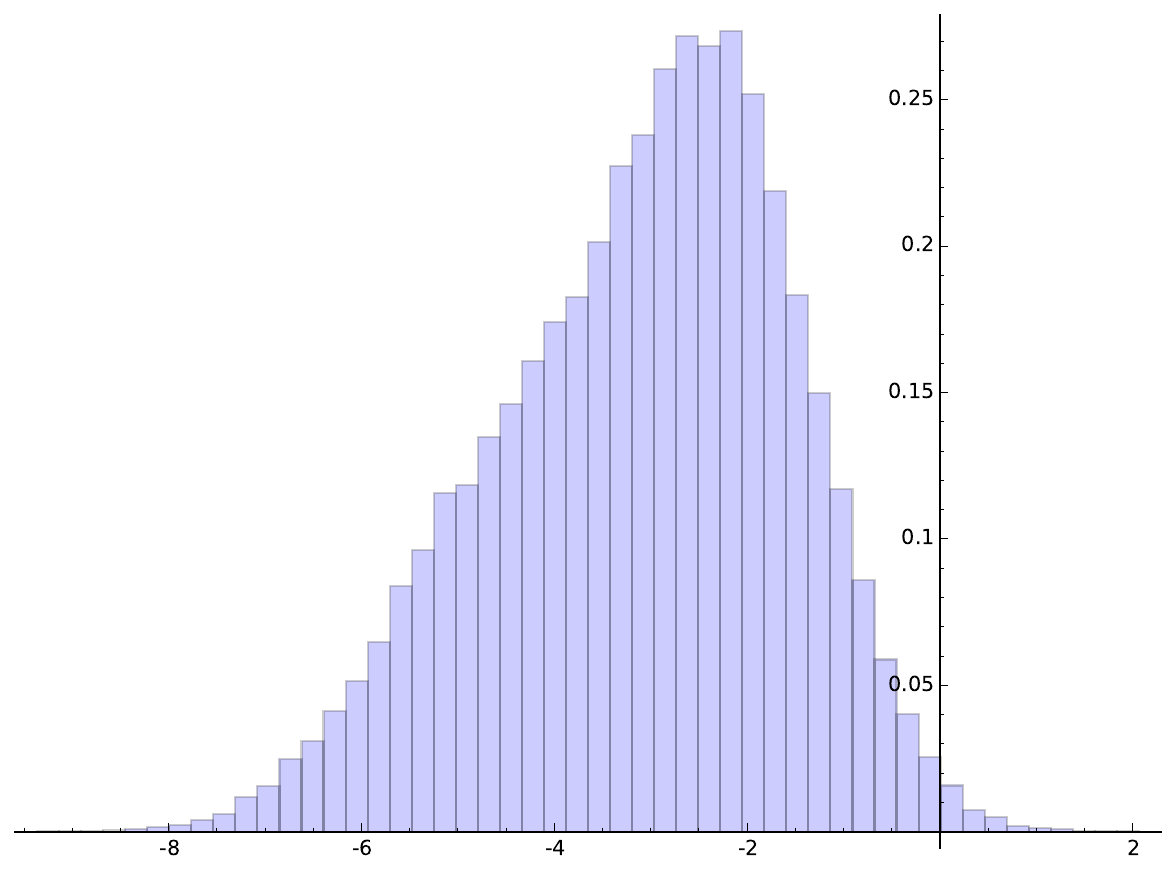}
\caption{Histogram for $r(h)$, as $h \equiv 0 \mod 3$ ranges over
  odd integers in $[500000,1000000]$. $(\mu,\sigma)=(-3.101265,1.529449)$.}
\end{figure}}
\newcommand{\classfigthree}{
\begin{figure}[H]
  \centering
\includegraphics[width=.6\textwidth]{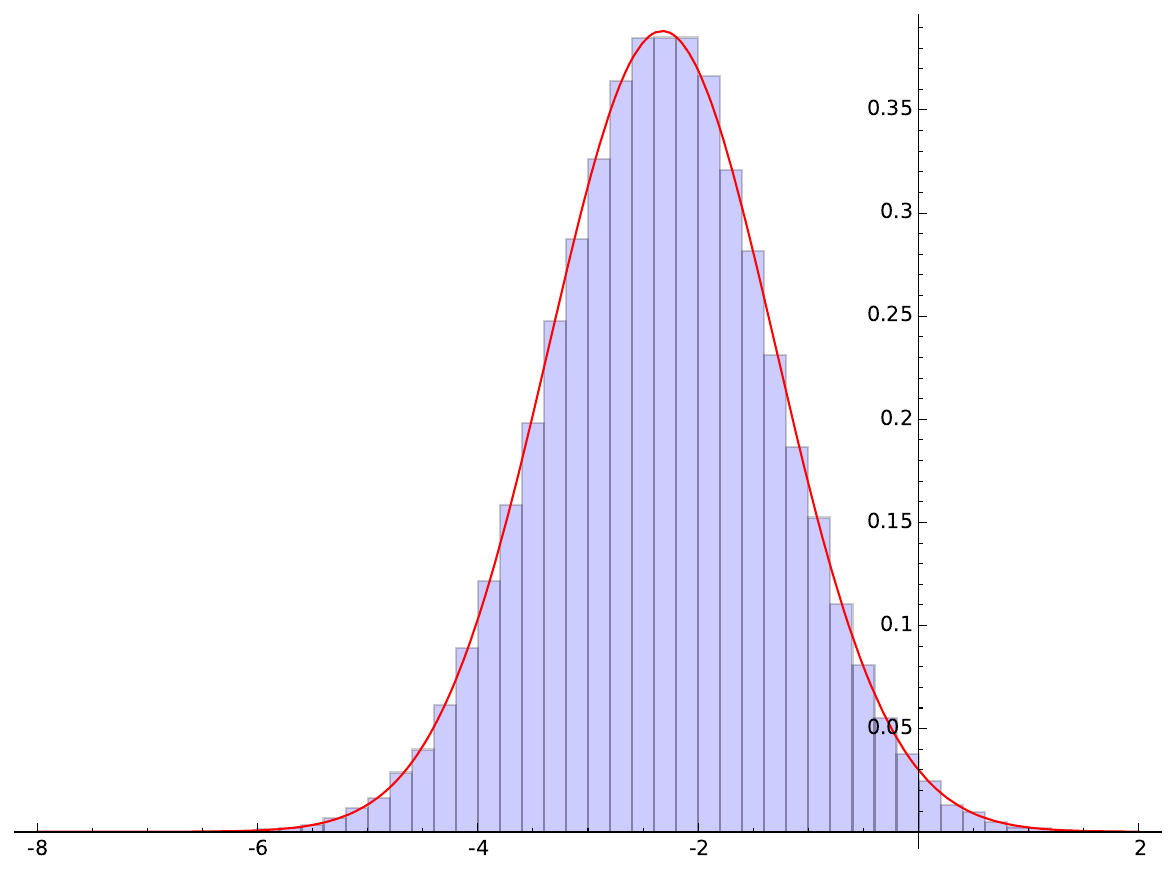}
\caption{Histogram for $r(h)$, for odd $h$ in (500000,1000000), $3 || h$.
$(\mu,\sigma)=(-2.326289,1.027387)$.}
\end{figure}}
\newcommand{\classfigfour}{
\begin{figure}[H]
  \centering
\includegraphics[width=.6\textwidth]{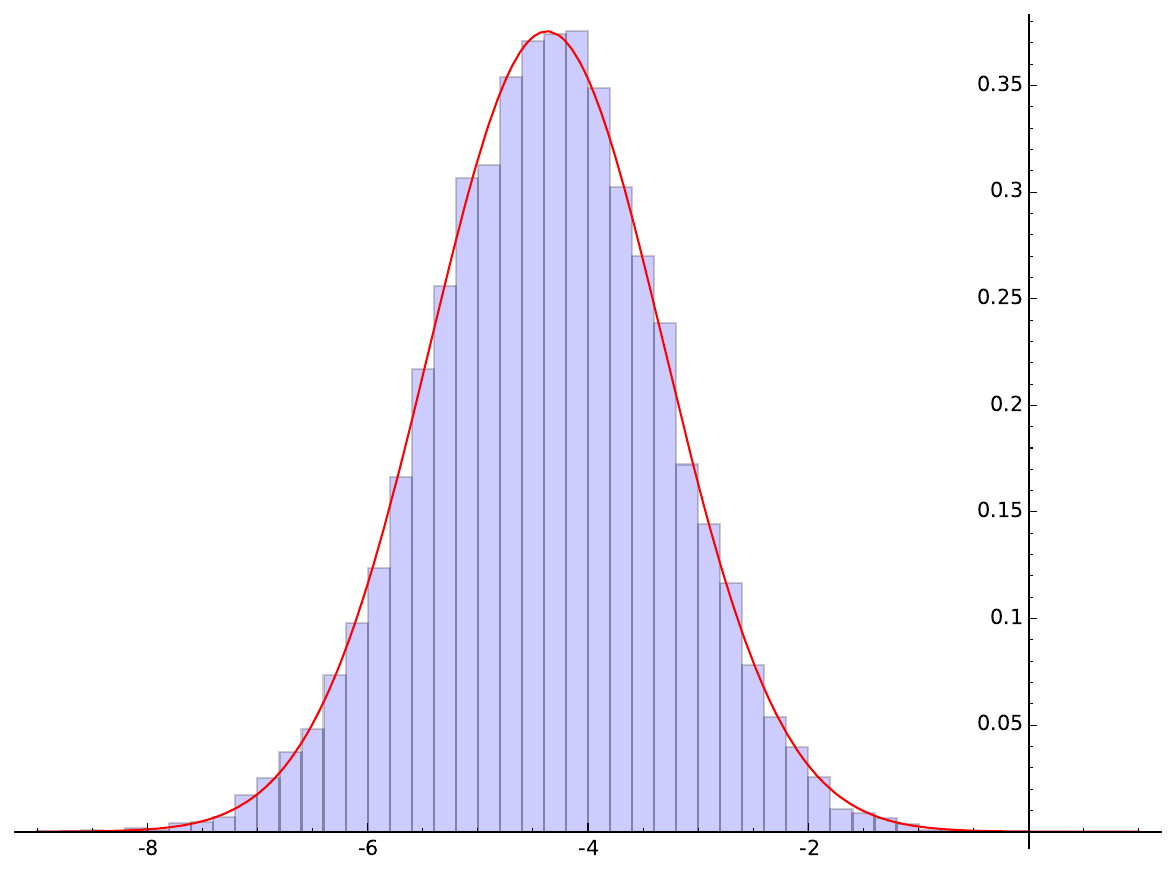}
\caption{Histogram for $r(h)$, for odd $h$ in (500000,1000000), $3^2 || h$.
$(\mu,\sigma)=(-4.372185,1.062480)$.}
\end{figure}}
\newcommand{\classfigfive}{
\begin{figure}[H]
  \centering
\includegraphics[width=.6\textwidth]{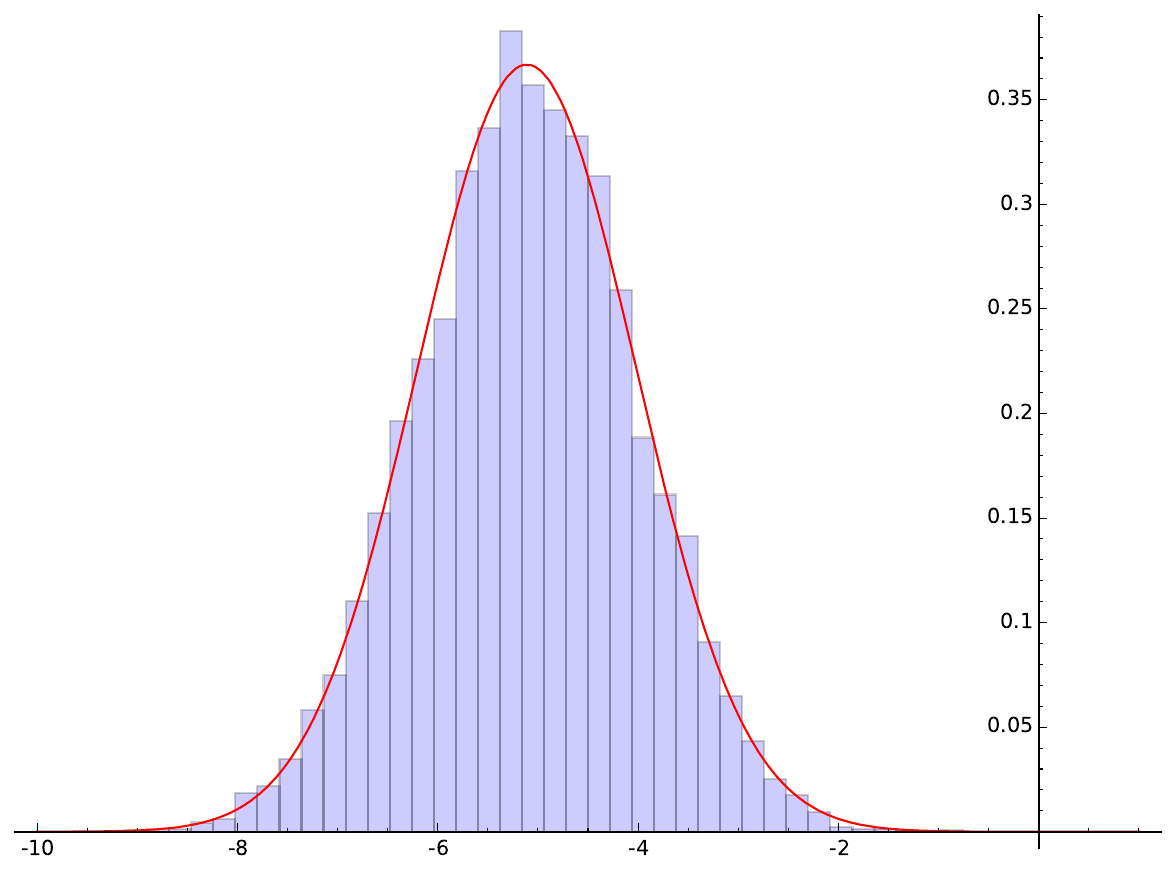}
\caption{Histogram for $r(h)$, for odd $h$ in (500000,1000000), $3^3 || h$.
$(\mu,\sigma)=(-5.110585,1.087463)$.}
\end{figure}}
\newcommand{\classfigsix}{
\begin{figure}[H]
  \centering
\includegraphics[width=.6\textwidth]{plot-6}
\caption{Histogram for $r(h)$, for odd $h$ in (500000,1000000), $3^4 || h$.
$(\mu,\sigma)=(-5.383597,1.058504)$.}
\end{figure}}
\newcommand{\classfigseven}{
\begin{figure}[H]
  \centering
\includegraphics[width=.6\textwidth]{plot-7}
\caption{Histogram for $r(h)$, for odd $h$ in (500000,1000000), $3^5 || h$.
$(\mu,\sigma)=(-5.427022,1.067687)$.}
\end{figure}}
\newcommand{\classfigeight}{
\begin{figure}[H]
  \centering
\includegraphics[width=.6\textwidth]{plot-8}
\caption{Histogram for $r(h)$, for odd $h$ in (500000,1000000), $3^6 || h$.
$(\mu,\sigma)=(-5.545782,1.082443)$.}
\end{figure}}
\newcommand{\classfigalloddtweak}{
\begin{figure}[H]
  \centering
\includegraphics[width=.6\textwidth]{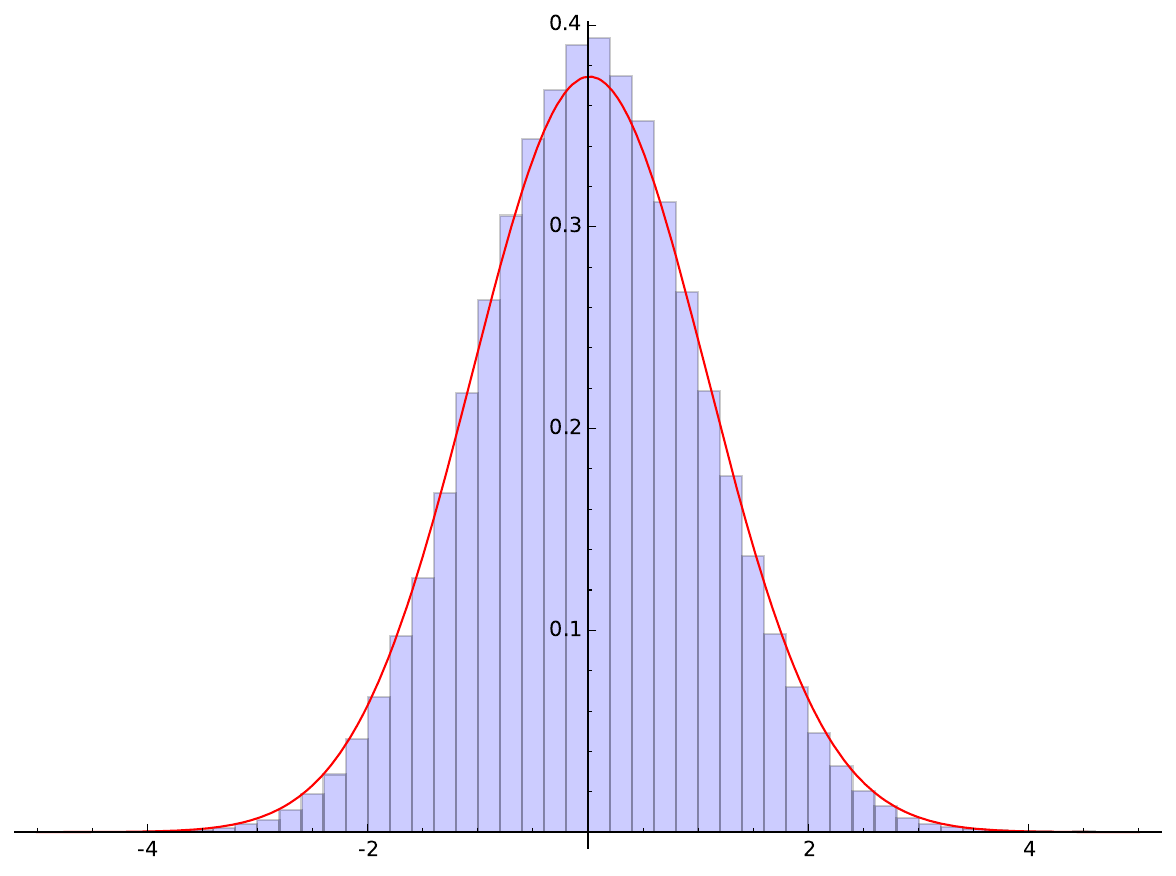}
\caption{Histogram for $r'(h)$, for all odd $h$ in (500000,1000000).
$(\mu,\sigma)=(0.013214,1.065277)$.}
\end{figure}}
\newcommand{\classfigalloddtweaktwo}{
\begin{figure}[H]
  \centering
\includegraphics[width=.6\textwidth]{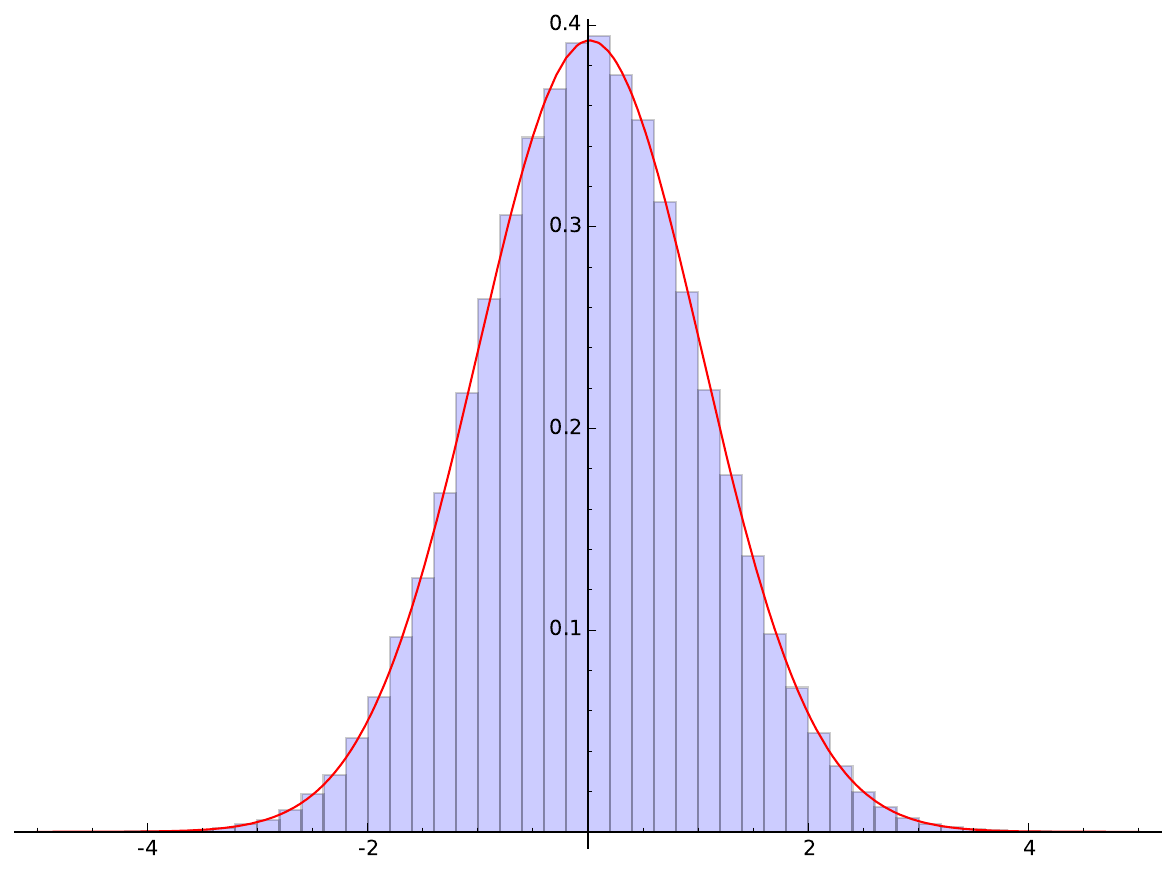}
\caption{Histogram for $r'(h)$, for all odd $h$ in (500000,1000000) and $h$ not divisible by 81.
$(\mu,\sigma)=(0.016292,1.016726)$.}
\end{figure}}
  
\classfigzero

Interestingly, the probability distribution appears to be bimodal.  A
closer inspection of the table above indicates a small positive bias
for $h$ that are divisible by three.  Separating out (odd) $h$
according to divisibility by three, or not, results in the following
two histograms:

\classfigtwo

\classfigone
The curve (red in color printouts and online) in the first plot is a
Gaussian probability density function with mean and standard deviation
fitted to the data --- the first plot appears to be Gaussian, whereas the
second clearly is not.

Also note that (after our normalization), the effect of three
divisibility is quite pronounced --- the shift in the mean value is of
order of magnitude a standard deviation.  

By further separating $h \equiv 0 \mod{ 3}$ into subsets according to
the exact power of three that divides $h$, we obtain distributions
that appear Gaussian; for comparison, we again plot a  (red) curve
giving  the
probability density function for a Gaussian 
random variable with the same mean and standard deviation as the
observed data.  (Note that there is a significant shift in the mean,
whereas the standard deviation is close to one.)

\classfigthree

\classfigfour

\classfigfive



The exact nature of this ``three divisibility bias'' is unclear, but
inspired by the slow convergence in the Davenport-Heilbronn
asymptotic\footnote{In fact, a negative second order correction to
  \eqref{3partofclassgrpasymp} of size $X^{5/6}$ was recently obtained
  by T. Taniguchi and F. Thorne
  \cite{taniguchi-thorne-secondary-term-cubic-fields} and also
  independently by M. Bhargava, A. Shankar and J. Tsimmerman
  \cite{bhargava-shankar-tsimerman-davenport-heilbronn-secondary}.} 
\begin{equation} \label{3partofclassgrpasymp}
\sum_{\substack{-X < d < 0\\ \text{$d$ fund. disc.}}} 
|H(d)[3]| \sim  C \cdot X
\end{equation}
(here $H(d)[3]$ denotes the $3$-torsion subgroup of $H(d)$) we can slightly adjust 
$\mf{c}(h)$ to remove most of this bias and obtain a more accurate
prediction $\operatorname{pred}'(h)$.  (Essentially we examine the
exact power of three divisibility of $h$ and adjust to the data, see
Section~\ref{sec:damp-three-divis} for more details.)  With this
adjustment, the fluctuations for
$$
r'(h) :=
\frac{\mc F(h)-\operatorname{pred}'(h)}{\sqrt{\operatorname{pred}'(h)}}
$$
(for the full set of odd $h$) is quite close to a Gaussian.
\classfigalloddtweak
However, compared to the fitted Gaussian, the histogram is slightly
more peaked, and has less mass in the tails.  If we remove integers
being divisible by $3^{4}$ this effect is reduced and we get an
improved fit to a Gaussian.

\classfigalloddtweaktwo

We now return to our discussion of the quantity $\mc{F}(G)$.  To make
precise what we mean by the \emph{shape} of an abelian $p$-group,
recall the bijection 
\[
\begin{split}
\left\{ \text{partitions of $n$} \right\} \; &\leftrightarrow \; \left\{ \text{abelian groups of order $p^n$} \right\}  \\
\gl = (n_1, n_2, \dots, n_r) &\mapsto G_\gl(p) := \bigoplus_{i=1}^r \mbz/p^{n_i}\mbz.
\end{split}
\]
Using \eqref{factorizationofmcFofG} in conjunction with Conjecture \ref{refinedsoundconjecture}, and evaluating each factor asymptotically, we are led to the following conjecture.  Given a partition
\[
\gl = (n_1, n_2, \dots, n_r), \quad n_1 \geq n_2 \geq \dots \geq n_r \geq 1, \quad n_1 + n_2 + \dots + n_r = n
\]
of $n$, define the \emph{cyclicity index} of $\gl$ by
\begin{equation} \label{defofcyclicityindex}
\cyc(\gl) := \sum_{i=1}^r (3-2i)n_i = n_1 - \sum_{i = 2}^r (2i-3)n_i.
\end{equation}
Note that $\cyc(\gl) \in [1-(n-1)^2, n]$ and $G_\gl(p)$ is cyclic if and only if $\cyc(\gl) = n$; thus $\cyc(\gl)$ provides a measure of how much $G_\gl(p)$ deviates from being cyclic.
\begin{Conjecture}  \label{mainconjecture}
Fix $n \in \mbn$ and a partition $\gl$ of $n$.  Then $\mc{F}(G_\gl(p)) > 0$ for infinitely many primes $p$ if and only if $\cyc(\gl) \geq 0$.
More precisely, if $\cyc(\gl) > 0$ then as $p \rightarrow \infty$ we have
\[
\mc{F}\left( G_\gl(p) \right) \sim \frac{\mf{C}}{n} \cdot \frac{p^{\cyc(\gl)}}{\log p},
\]
where $\mf{C}$ is as in \eqref{defofmfC}.  If $\cyc(\gl) = 0$ then as
$x \rightarrow \infty$ we have
\[
\sum_{p \leq x \atop p \text{ prime}} \mc{F}\left( G_\gl(p) \right) \sim \frac{\mf{C}}{n} \cdot \frac{x}{(\log x)^2}.
\]
If $\cyc(\gl) < 0$ then
\[
p \gg_\gl 1 \; \implies \; \mc{F}(G_\gl(p)) = 0.
\]
\end{Conjecture}

\begin{Definition}
  We say that a partition $\gl$ of $n$ is
  \emph{attainable} if $\cyc(\gl) \geq 0$.
\end{Definition}
Thus, Conjecture \ref{mainconjecture} implies that $G_\gl(p)$ occurs as a class group for infinitely many primes $p$ if and only if $\gl$ is attainable.  What is the relative proportion of attainable partitions among all partitions? 
The following table suggests that the relative proportion decreases
with $n$. 
\[
\begin{array}{|c||c|c|c|c|c|c|c|c|c|c|c|} \hline n & 4 & 5 & 6 & 7 & 8 & 9 & 10 & 11 & 12 & \dots & 100 \\ 
\hline 
\hline
\# \{ \text{attainable partitions of $n$}\} & 3 & 3 & 5 & 5 & 7 & 7 & 9 & 9 & 13 & \dots & 4742 \\
\hline
\# \{ \text{partitions of $n$}\} & 5 & 7 & 11 & 15 & 22 & 30 & 42 & 56 & 77 & \dots & 190\,569\,292 \\
\hline
\text{Ratio} & 0.6 & 0.43 & 0.45 & 0.33 & 0.32 & 0.23 & 0.21 & 0.16 & 0.17 & \dots & 0.000025\\
\hline
\end{array}
\]
Our next theorem confirms this.
\begin{Theorem} \label{camstheorem}
For a positive integer $n$, we have
\[
\frac{\# \{ \text{attainable partitions of $n$} \} }{\#\{ \text{partitions of $n$} \}} \; \ll \; n^{3/4}e^{(2-\sqrt{\frac{2}{3}}\pi)\sqrt{n}}.
\]
In particular, 
\[
\lim_{n \rightarrow \infty} \frac{\# \{ \text{attainable partitions of $n$} \} }{\#\{ \text{partitions of $n$} \}} = 0.
\]
\end{Theorem}

\subsection{Numerical investigations of attainable  groups}
\label{sec:numerical-evidence}

For families of $p$-groups with $\cyc(\gl)>0$, we expect that many (if
not all) groups should occur; in fact $\mathcal F(G_\lambda(p))$
should grow with $p$.  On the other hand, there should be very few (if
any at all) in case $\cyc(\gl)<0$ --- we call these groups
``sporadic''. 

In this section, we present numerical evidence supporting Conjecture
\ref{mainconjecture} based on our computer computation of $\mathcal
F(G)$, conditional on GRH, for all $p$-groups $G$ of odd size at most
$10^6$. (See Section \ref{numerics} for details regarding the
computation.)

\subsubsection {Numerics on $\mathcal F(G_\lambda(p))$}
\label{sec:numer-mathc-fg_l}
We give in the tables below\footnote{The complete list of all $\mathcal F(G_\lambda(p))$ is given in \cite{data-FG}, and a complete list of all corresponding discriminants $d$ and groups $H(d)$ is given in \cite{data-d-Hd}.}
the value of $\mathcal F(G_\lambda(p))$ (conditional on GRH) for each
odd prime $p$ and each partition $\lambda$ of some $n\geq 3$, such
that $\abs{G_\lambda(p)}<10^6$.  To be precise: The second column in
each table contains all partitions of $n$ for some fixed $n$, ordered
by decreasing cyclicity index $\cyc(\lambda)$, which itself is given
in the leftmost column. The top row contains a list of all primes $p$
such that $p^n<10^6$, and under each $p$ we list the values of
$\mathcal F(G_\lambda(p))$ corresponding to the partition $\lambda$ in
the same row. Whenever a partition is omitted from a table, then it is
implied that all omitted values of $\mathcal F(G_\lambda(p))$ are
zero.  Groups occuring in rows corresponding to negative cyclicity
index (``sporadic groups'') are star/bold-marked for emphasis (also
see Section~\ref{sec:sporadic-groups}.)

\begin{small}
\newcommand{\twocols}[2]{{
  \begin{center}
    \renewcommand{\arraystretch}{0} 
    \begin{tabular}{cc}
      \phantom{\rule{\dimexpr 0.5\linewidth-2\tabcolsep}{0pt}} & \phantom{\rule{\dimexpr 0.5\linewidth-2\tabcolsep}{0pt}} \\ 
      $#1$ & $#2$ \\
    \end{tabular}
  \end{center}
}}

\newcommand{\cyclambda}{\cyc(\lambda)}

\begin{center}
  $
  \setlength{\arraycolsep}{3pt}
  \begin{array}{|r|l||rrrrrrrrrrrr|}
  \hline
  \cyclambda & \lambda   &  p=3 &   5 &   7 &   11 &   13 &   17 &   19 &    23 &    29 &    31 &    37 &    41  \\
  \hline
     3 & (3)       & 88 & 279 & 607 & 1856 & 2904 & 5797 & 7963 & 12958 & 24407 & 29201 & 46981 & 62327  \\
     1 & (2, 1)    &  5 &  11 &  13 &   19 &   17 &   25 &   22 &    29 &    35 &    26 &    39 &    37  \\
    -3 & (1, 1, 1) &  0 &   0 &   0 &    0 &    0 &    0 &    0 &     0 &     0 &     0 &     0 &     0  \\
  \hline
  \hline
  \cyclambda & \lambda   &     p=43 &    47 &     53 &     59 &     61 &     67 &     71 &     73 &     79 &     83 &     89 &     97 \\
  \hline
     3 & (3)       &  71617 & 91690 & 127190 & 170444 & 186988 & 242464 & 283998 & 306567 & 382770 & 438976 & 533751 & 678610 \\
     1 & (2, 1)    &     39 &    29 &     46 &     48 &     57 &     55 &     60 &     66 &     51 &     73 &     66 &     69 \\
    -3 & (1, 1, 1) &      0 &     0 &      0 &      0 &      0 &      0 &      0 &      0 &      0 &      0 &      0 &      0 \\
  \hline
  \end{array}
  $
\end{center}

\begin{center}
  $
  \setlength{\arraycolsep}{3pt}
  \begin{array}{|r|l||rrrrrrrrrr|}
  \hline
  \cyclambda & \lambda      &   p=3 &    5 &    7 &    11 &    13 &    17 &     19 &     23 &     29 &     31 \\
  \hline
     4 & (4)          & 206 & 1093 & 3404 & 16290 & 29496 & 77693 & 116710 & 233027 & 548392 & 701408 \\
     2 & (3, 1)       &  19 &   47 &   71 &   146 &   197 &   244 &    343 &    480 &    644 &    779 \\
     0 & (2, 2)       &   3 &    0 &    0 &     0 &     2 &     1 &      2 &      1 &      0 &      1 \\
    -2 & (2, 1, 1)    &   0 &    0 &    0 &     0 &     0 &     0 &      0 &      0 &      0 &      0 \\
    -8 & (1, 1, 1, 1) &   0 &    0 &    0 &     0 &     0 &     0 &      0 &      0 &      0 &      0 \\
  \hline
  \end{array}
  $
\end{center}

\twocols{
    \setlength{\arraycolsep}{3pt}
    \begin{array}{|r|l||rrrrr|}
    \hline
    \cyclambda & \lambda         &   p=3 &    5 &     7 &     11 &     13 \\
    \hline
       5 & (5)             & 549 & 4610 & 19430 & 147009 & 314328 \\
       3 & (4, 1)          &  56 &  218 &   444 &   1347 &   1894 \\
       1 & (3, 2)          &   8 &    5 &     8 &     13 &      9 \\
      -1 & (3, 1, 1)       &   0 &    {\bf 1*} &     0 &      0 &      0 \\
      -3 & (2, 2, 1)       &   0 &    0 &     0 &      0 &      0 \\
      -7 & (2, 1, 1, 1)    &   0 &    0 &     0 &      0 &      0 \\
     -15 & (1, 1, 1, 1, 1) &   0 &    0 &     0 &      0 &      0 \\
    \hline
    \end{array}
  }{
    \setlength{\arraycolsep}{3pt}
    \begin{array}{|r|l||rrr|}
    \hline
    \cyclambda & \lambda            &    p=3 &     5 &      7 \\
    \hline
       6 & (6)                & 1512 & 19469 & 116278 \\
       4 & (5, 1)             &  177 &  1024 &   2887 \\
       2 & (4, 2)             &   18 &    37 &     58 \\
       0 & (4, 1, 1)          &    0 &     3 &      0 \\
       0 & (3, 3)             &    2 &     2 &      3 \\
      -2 & (3, 2, 1)          &    0 &     0 &      0 \\
      -6 & (3, 1, 1, 1)       &    0 &     0 &      0 \\
      -6 & (2, 2, 2)          &    0 &     0 &      0 \\
      -8 & (2, 2, 1, 1)       &    0 &     0 &      0 \\
     -14 & (2, 1, 1, 1, 1)    &    0 &     0 &      0 \\
     -24 & (1, 1, 1, 1, 1, 1) &    0 &     0 &      0 \\
    \hline
    \end{array}
  }

\twocols{
    \setlength{\arraycolsep}{3pt}
    \begin{array}{|r|l||rrr|}
    \hline
    \cyclambda & \lambda               &    p=3 &     5 &      7 \\
    \hline
       7 & (7)                   & 3881 & 86038 & 711865 \\
       5 & (6, 1)                &  571 &  4259 &  17057 \\
       3 & (5, 2)                &   58 &   177 &    372 \\
       1 & (5, 1, 1)             &    7 &     7 &      6 \\
       1 & (4, 3)                &    8 &    11 &      7 \\
      -1 & (4, 2, 1)             &    {\bf 1*} &     0 &      0 \\
      -3 & (3, 3, 1)             &    {\bf 1*} &     0 &      0 \\
      -5 & (4, 1, 1, 1)          &    0 &     0 &      0 \\
      -5 & (3, 2, 2)             &    0 &     0 &      0 \\
      -7 & (3, 2, 1, 1)          &    0 &     0 &      0 \\
     -11 & (2, 2, 2, 1)          &    0 &     0 &      0 \\
     -13 & (3, 1, 1, 1, 1)       &    0 &     0 &      0 \\
     -15 & (2, 2, 1, 1, 1)       &    0 &     0 &      0 \\
     -23 & (2, 1, 1, 1, 1, 1)    &    0 &     0 &      0 \\
     -35 & (1, 1, 1, 1, 1, 1, 1) &    0 &     0 &      0 \\
    \hline
    \end{array}
  }{
    \setlength{\arraycolsep}{3pt}
    \begin{array}{|r|l||rr|}
    \hline
    \cyclambda & \lambda                  &     p=3 &      5 \\
    \hline
       8 & (8)                      & 10712 & 379751 \\
       6 & (7, 1)                   &  1585 &  18956 \\
       4 & (6, 2)                   &   180 &    719 \\
       2 & (6, 1, 1)                &    18 &     30 \\
       2 & (5, 3)                   &    15 &     24 \\
       0 & (5, 2, 1)                &     4 &      1 \\
       0 & (4, 4)                   &     2 &      0 \\
      -2 & (4, 3, 1)                &    {\bf 1*} &      0 \\
      -4 & (5, 1, 1, 1)             &     0 &      0 \\
     \vdots & \vdots    &     \vdots &      \vdots \\
     -48 & (1, 1, 1, 1, 1, 1, 1, 1) &     0 &      0 \\
    \hline
    \end{array}
  }

\twocols{
    \setlength{\arraycolsep}{3pt}
    \begin{array}{|r|l||r|}
    \hline
    \cyclambda & \lambda                     &     p=3 \\
    \hline
       9 & (9)                         & 28308 \\
       7 & (8, 1)                      &  4516 \\
       5 & (7, 2)                      &   454 \\
       3 & (7, 1, 1)                   &    42 \\
       3 & (6, 3)                      &    54 \\
       1 & (6, 2, 1)                   &    10 \\
       1 & (5, 4)                      &     4 \\
      -1 & (5, 3, 1)                   &     {\bf 1*} \\
      -3 & (6, 1, 1, 1)                &     0 \\
     \vdots & \vdots    &     \vdots  \\
     -63 & (1, 1, 1, 1, 1, 1, 1, 1, 1) &     0 \\
    \hline
    \end{array}
  }{
    \setlength{\arraycolsep}{3pt}
    \begin{array}{|r|l||r|}
    \hline
    \cyclambda & \lambda                        &     p=3 \\
    \hline
      10 & (10)                           & 78657 \\
       8 & (9, 1)                         & 12433 \\
       6 & (8, 2)                         &  1446 \\
       4 & (8, 1, 1)                      &   160 \\
       4 & (7, 3)                         &   167 \\
       2 & (7, 2, 1)                      &    16 \\
       2 & (6, 4)                         &    14 \\
       0 & (6, 3, 1)                      &     1 \\
       0 & (5, 5)                         &     0 \\
     \vdots & \vdots    &     \vdots  \\
     -80 & (1, 1, 1, 1, 1, 1, 1, 1, 1, 1) &     0 \\
    \hline
    \end{array}
  }

\twocols{
    \setlength{\arraycolsep}{3pt}
    \begin{array}{|r|l||r|}
    \hline
    \cyclambda & \lambda                           &      p=3 \\
    \hline
      11 & (11)                              & 216520 \\
       9 & (10, 1)                           &  35544 \\
       7 & (9, 2)                            &   3880 \\
       5 & (9, 1, 1)                         &    437 \\
       5 & (8, 3)                            &    460 \\
       3 & (8, 2, 1)                         &     58 \\
       3 & (7, 4)                            &     49 \\
       1 & (7, 3, 1)                         &     10 \\
       1 & (6, 5)                            &      9 \\
      -1 & (8, 1, 1, 1)                      &      0 \\
      -1 & (7, 2, 2)                         &     {\bf 1*} \\
      -1 & (6, 4, 1)                         &     {\bf 1*} \\
      -3 & (7, 2, 1, 1)                      &      0 \\
     \vdots & \vdots    &     \vdots  \\
     -99 & (1, 1, 1, 1, 1, 1, 1, 1, 1, 1, 1) &      0 \\
    \hline
    \end{array}
  }{
    \setlength{\arraycolsep}{3pt}
    \begin{array}{|r|l||r|}
    \hline
    \cyclambda & \lambda                              &      p=3 \\
    \hline
      12 & (12)                                 & 603525 \\
      10 & (11, 1)                              &  98421 \\
       8 & (10, 2)                              &  10988 \\
       6 & (10, 1, 1)                           &   1291 \\
       6 & (9, 3)                               &   1265 \\
       4 & (9, 2, 1)                            &    220 \\
       4 & (8, 4)                               &    133 \\
       2 & (8, 3, 1)                            &     26 \\
       2 & (7, 5)                               &     17 \\
       0 & (9, 1, 1, 1)                         &      2 \\
       0 & (8, 2, 2)                            &      1 \\
       0 & (7, 4, 1)                            &      1 \\
       0 & (6, 6)                               &      2 \\
      -2 & (8, 2, 1, 1)                         &     {\bf 1*} \\
      -2 & (7, 3, 2)                            &      0 \\
     \vdots & \vdots    &     \vdots  \\
    -120 & (1, 1, 1, 1, 1, 1, 1, 1, 1, 1, 1, 1) &      0 \\
    \hline
    \end{array}
  }
\end{small}

Below we plot, for $p$ ranging over {\em odd primes}, observed values
$\mathcal F(G_\lambda(p))$ (black dots) versus predicted values
$P(G_\lambda(p))\cdot \operatorname{pred}(\abs{G_\lambda(p)})$ (red
dashed lines) for various partitions $\lambda$ with positive cyclicity
index $\cyc(\lambda)>0$.

{
\noindent
\setlength{\tabcolsep}{0pt}
\begin{tabular}{cc}
  \includegraphics[width=.5\textwidth]{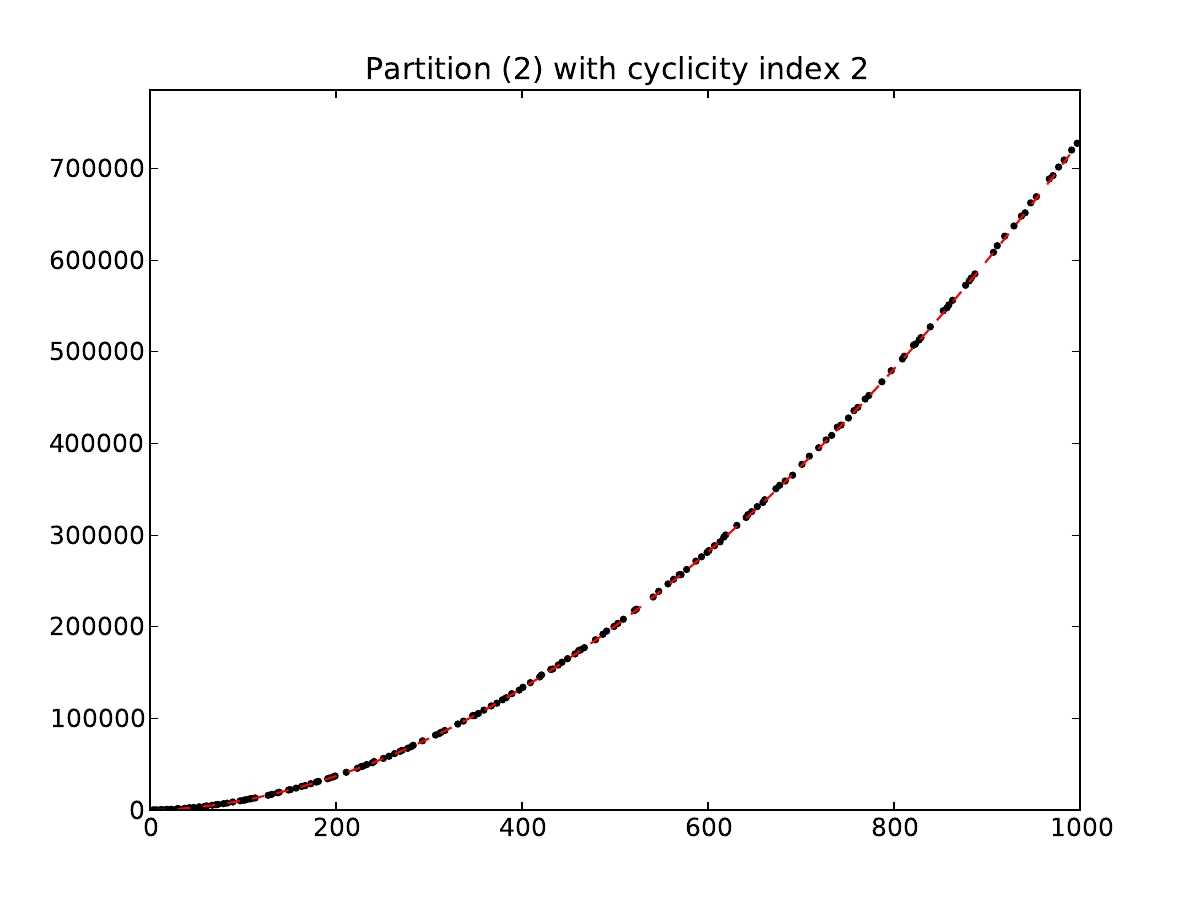} &
  \includegraphics[width=.5\textwidth]{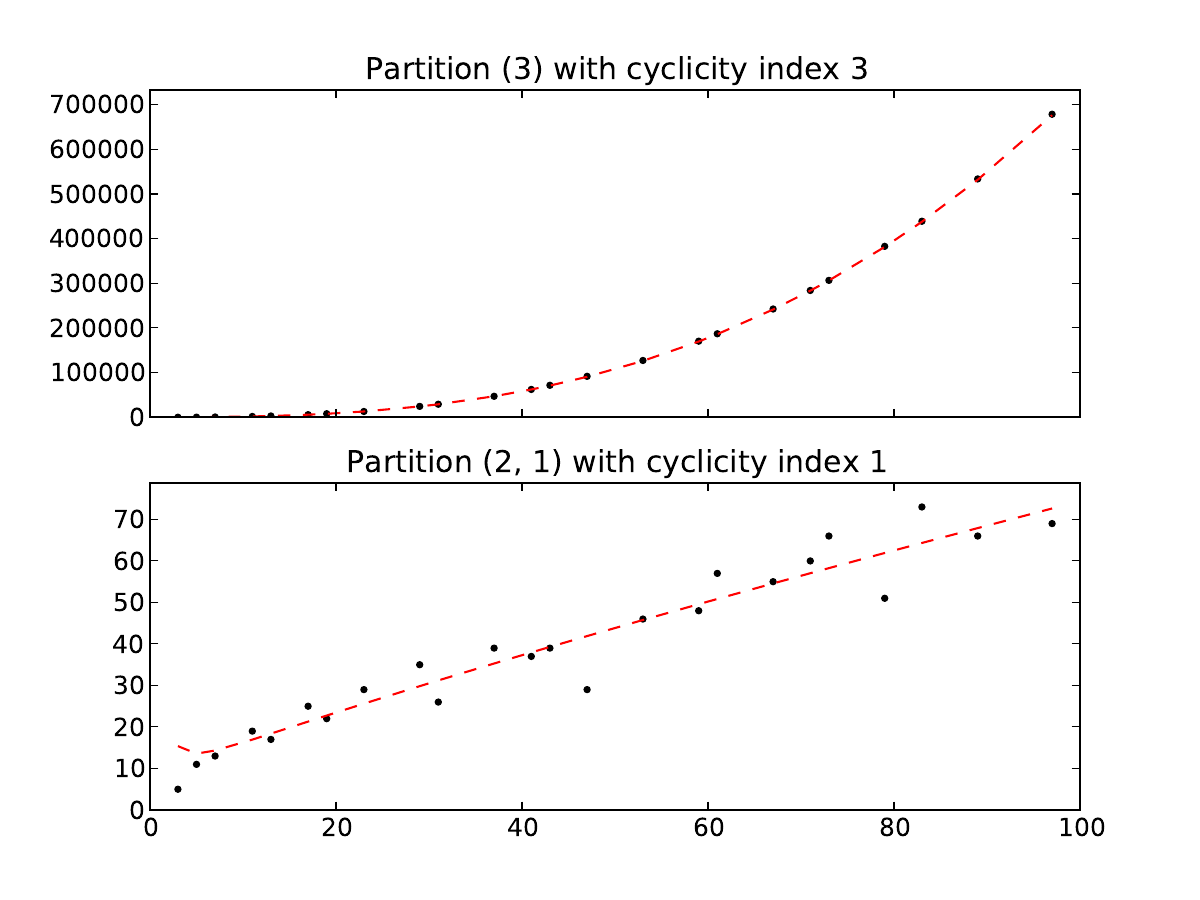} \\
  Partitions of 2 & Partitions of 3
\end{tabular}
}

{
\noindent
\setlength{\tabcolsep}{0pt}
\begin{tabular}{cc}
  \includegraphics[width=.5\textwidth]{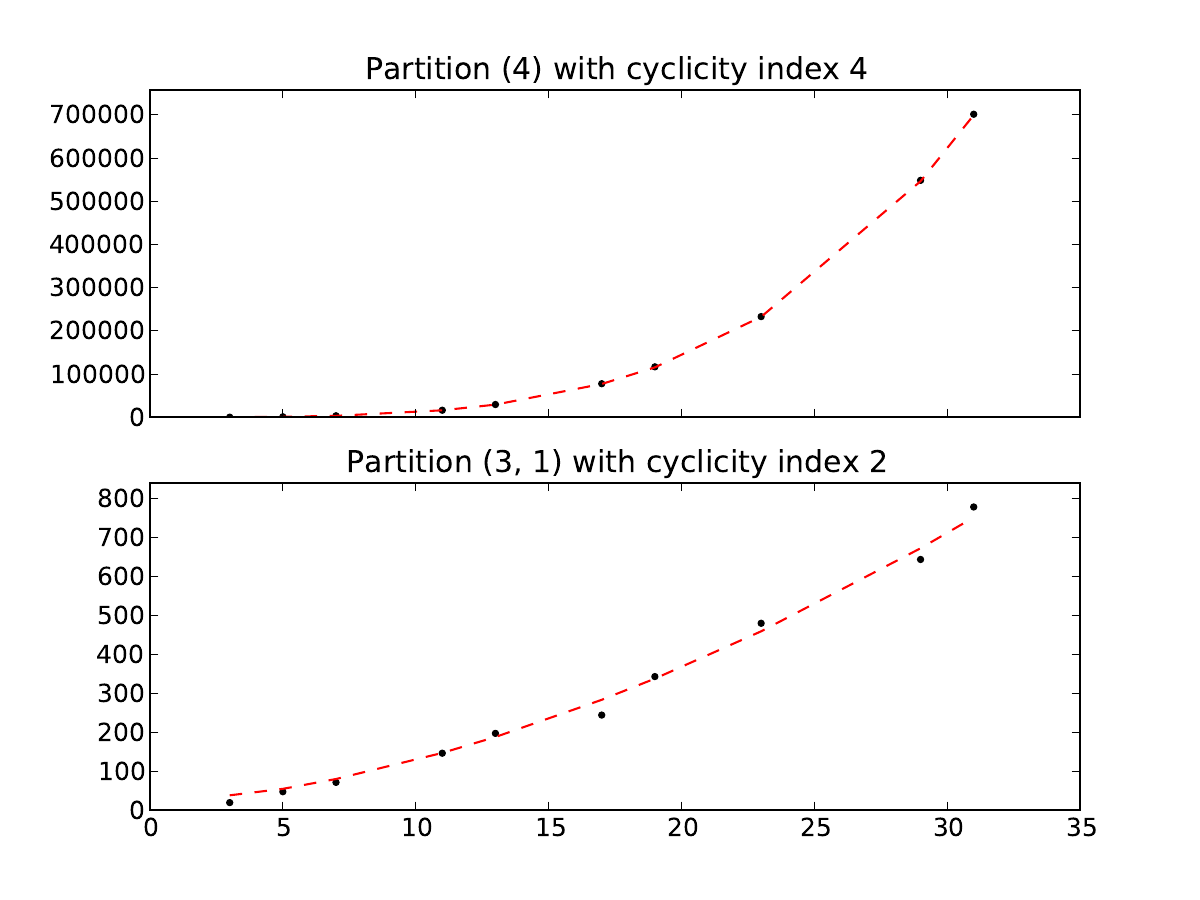} &
  \includegraphics[width=.5\textwidth]{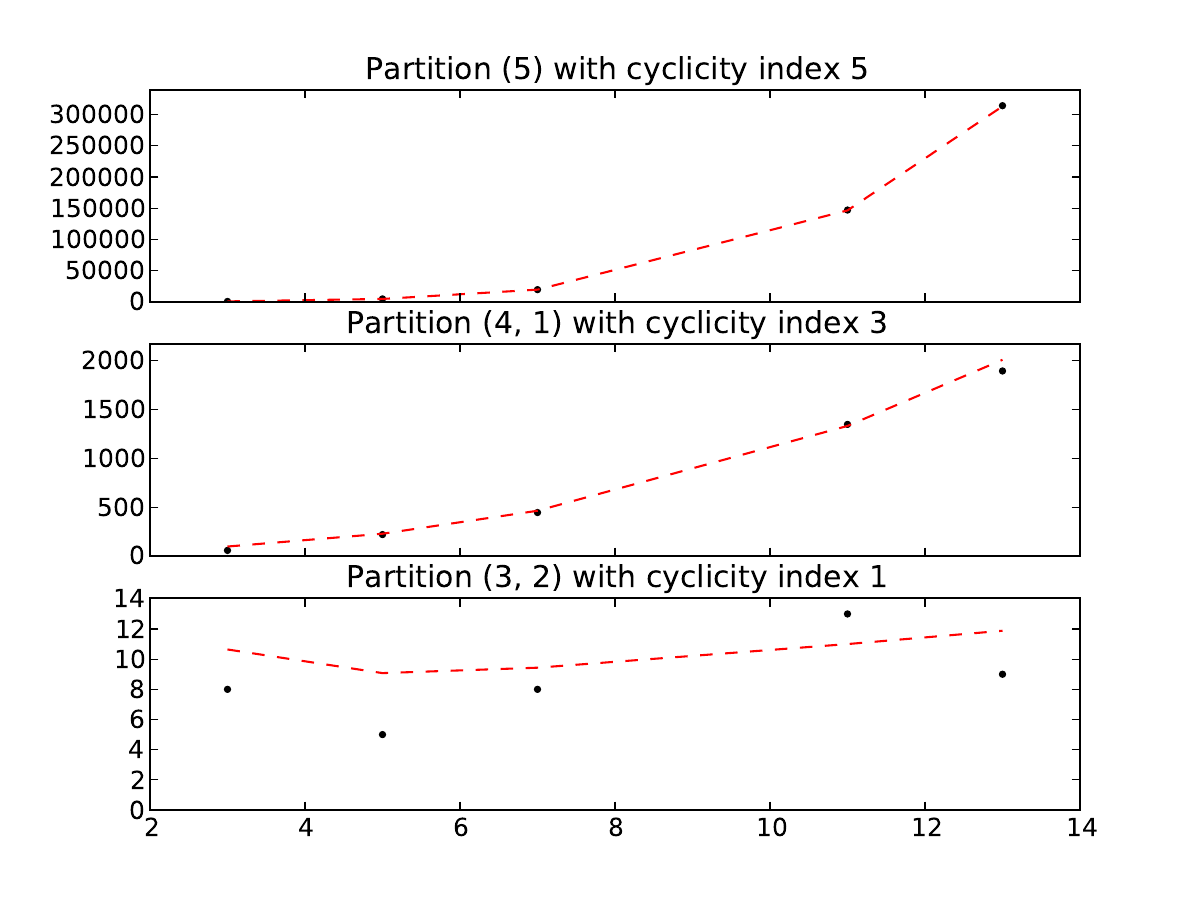} \\
  Partitions of 4 & Partitions of 5
\end{tabular}
}

We remark that each vanishing entry in the tables above
corresponds to a ``missing''
group. In particular we see that the group $(\Z/p\Z)^3$ does not
appear as the class group of a quadratic imaginary field for any prime
$2<p<100$.  Based on a combination of heuristics and numerics, it is 
reasonable to conjecture that $(\mbz/p\mbz)^n$ does not occur
for any odd prime $p$ and any $n \geq 3$.

\begin{Conjecture}\label{NoElemAb}
For $p$ odd, there are no elementary abelian $p$-groups of rank at
least 3 which occur as the class group of an imaginary quadratic
field.
\end{Conjecture}
Indeed, by \eqref{factorizationofmcFofG} and Conjecture \ref{refinedsoundconjecture}, together with the observed (GRH-conditional) fact that no $(\mbz/p\mbz)^n$ occurs as an imaginary quadratic class group for $p^n \leq 10^6$, we may bound the expected number
of counterexamples by
\[
\mf{C} \sum_{{\begin{substack} { p, n \geq 3  \\  p^n > 10^6 } \end{substack}}} \frac{\mf{c}(p^n)}{n p^{n^2 - 2n} \log p} \; \leq \; \mf{C} \cdot \prod_{i=1}^\infty \left( 1 - \frac{1}{2^i} \right)^{-1}  \sum_{{\begin{substack} { p, n \geq 3 \\  p^n > 10^6 } \end{substack}}} \frac{1}{n p^{n^2 - 2n} \log p}.
\]
Since the right-hand sum can then be bounded by $10^{-4}$, Conjecture
\ref{NoElemAb} is heuristically justified. 

Finally, we observe that none of the groups $G_\lambda(p)$ of
odd size $<10^6$ with $\cyc(\lambda)>0$ are missing. 

\subsubsection{Sporadic groups in negative cyclicity
  index case} 
\label{sec:sporadic-groups}
As just indicated with bold/star-marks in the tables, each of the
groups
\[
\begin{split}
\frac{\mbz}{5^3\mbz} \times \left( \frac{\mbz}{5\mbz} \right)^2, \quad\quad &\frac{\mbz}{3^4\mbz} \times \frac{\mbz}{3^2\mbz} \times \frac{\mbz}{3\mbz}, \quad\quad \left( \frac{\mbz}{3^3\mbz} \right)^2 \times \frac{\mbz}{3\mbz}, \\
\frac{\mbz}{3^4\mbz} \times \frac{\mbz}{3^3\mbz} \times \frac{\mbz}{3\mbz}, \quad\quad &\frac{\mbz}{3^5\mbz} \times \frac{\mbz}{3^3\mbz} \times \frac{\mbz}{3\mbz}, \quad\quad \frac{\mbz}{3^7\mbz} \times \left( \frac{\mbz}{3^2\mbz} \right)^2, \\
\frac{\mbz}{3^6\mbz} \times \frac{\mbz}{3^4\mbz} \times \frac{\mbz}{3\mbz}, \quad\quad &\frac{\mbz}{3^8\mbz} \times \frac{\mbz}{3^2\mbz} \times \left( \frac{\mbz}{3\mbz} \right)^2
\end{split}
\]
occurs exactly once as an imaginary quadratic class group, even
though $\cyc(\gl) < 0$ for each corresponding partition $\gl$.  From
the point of view of Conjecture \ref{mainconjecture}, these examples
may be regarded as ``sporadic,'' since conjecturally they do not
belong to an infinite family.

\subsubsection {Zero cyclicity index --- the family $\mathcal F((\Z/p\Z)^2)$}
The case of  $\cyc(\gl) = 0$ is intermediate in the sense that
infinitely many groups in the family should occur, and infinitely many
should not.  Here the data is quite limited, and we restrict ourselves
to the family $G=(\Z/p\Z)^2$.
The following table contains all odd primes $p$ such that $p^2<10^6$,
grouped according to the value of $\mathcal F((\Z/p\Z)^2)$, assuming
GRH. 
\begin{center}
\begin{tabular}{|r|p{13cm}|}
  \hline
  $n$ & All primes $p<1000$ such that $\mathcal F((\Z/p\Z)^2)=n$ \\
  \hline
  \hline
  0 & 11, 19, 37, 79, 89, 97, 103, 139, 151, 167, 181, 191, 193, 227, 229, 233, 241, 251, 271, 281, 283, 311, 313, 317, 349, 353, 359, 383, 401, 409, 433, 443, 463, 467, 479, 491, 499, 523, 563, 571, 587, 601, 619, 631, 643, 673, 701, 709, 733, 757, 769, 787, 809, 829, 877, 887, 907, 919, 929, 947, 953, 977, 983 \\
  \hline
  1 & 3, 17, 23, 41, 43, 47, 61, 67, 73, 107, 109, 113, 127, 131, 137, 157, 163, 173, 179, 199, 239, 257, 263, 269, 277, 293, 307, 331, 337, 347, 367, 373, 379, 397, 419, 439, 457, 487, 503, 509, 521, 547, 557, 577, 599, 613, 617, 641, 653, 659, 677, 683, 691, 719, 727, 739, 743, 761, 797, 811, 821, 823, 839, 853, 857, 859, 863, 881, 937, 941, 971, 991, 997 \\
  \hline
  2 & 5, 7, 29, 31, 53, 59, 71, 83, 101, 197, 211, 223, 389, 431, 449, 461, 569, 593, 607, 647, 661, 827, 883, 911 \\
  \hline
  3 & 149, 421, 541, 751, 967 \\
  \hline
  4 & 773 \\
  \hline
  5 & 13 \\
  \hline
\end{tabular}
\end{center}
The  limited data  seems to support intermediate behaviour.

One may ask how well our prediction of $\mathcal F(G)$, using equation
\eqref{factorizationofmcFofG}, holds up. The following graph compares
the cumulants of the predictions with the observations.

\begin{figure}[H]
  \centering
\includegraphics[width=.6\textwidth]{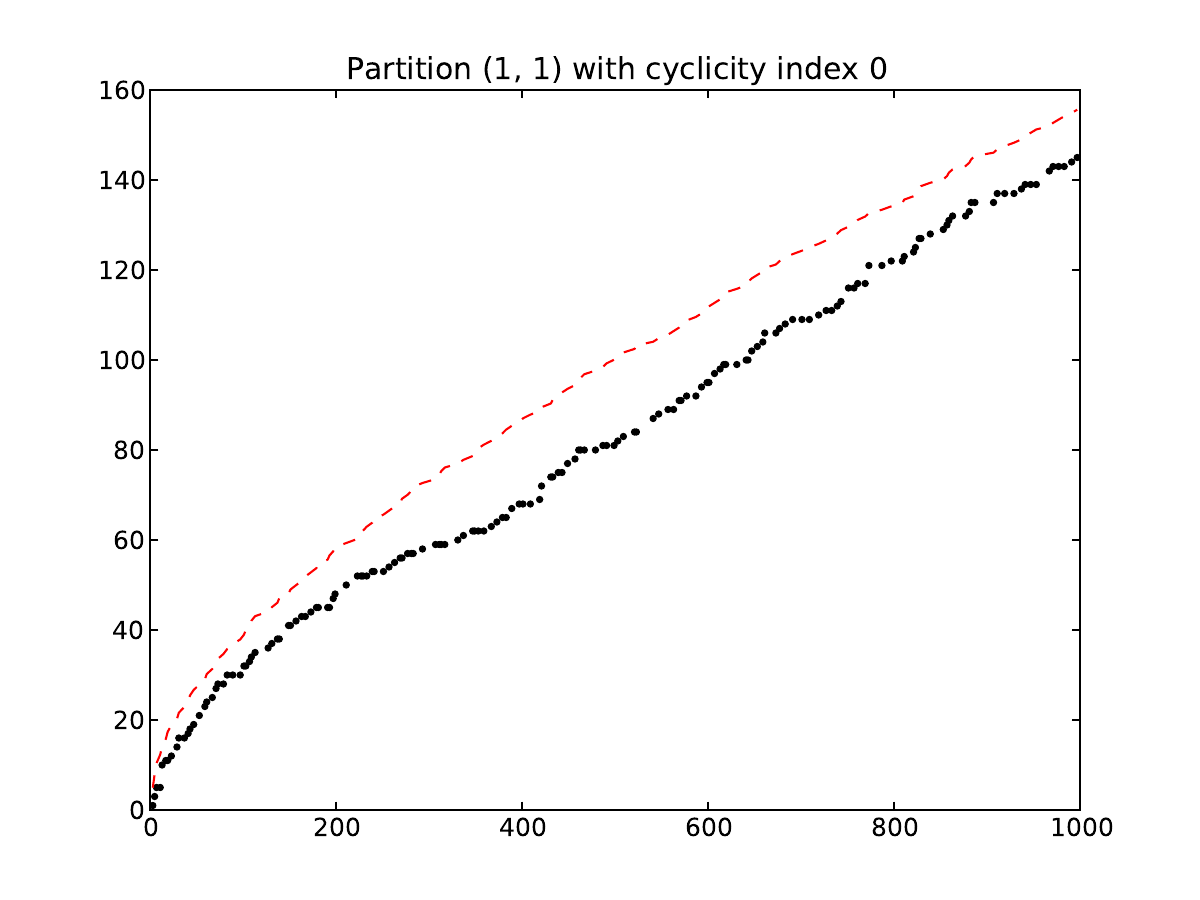}
\caption{
    {\em Cumulative} observed values $\sum_{p<x}\mathcal F(G_{(1,1)}(p))$ (black dots) compared to
    {\em cumulative} predicted values $\sum_{p<x}P(G_{(1,1)}(p))\operatorname{pred}(p^2)$ (red dashed line), for each prime $x<1000$.}
\end{figure}

\subsection{Related work}
\label{sec:discussion}

Certain classes of finite abelian groups are already known {\em not} to occur as imaginary quadratic class
groups.  For instance, letting $H(d)[n]$ denote the $n$-torsion
subgroup of $H(d)$, it is known that
$$
|H(d)[2]| \ll |H(d)|^{o(1)}
$$
(this is essentially genus theory together with Siegel's lower bound
on the class number; if $H(d)$ has two rank $r$, then $d$ has at least
$r-1$ distinct prime factors).  In particular, for any fixed $\epsilon > 0$ there are only finitely many imaginary quadratic class groups $H(d)$ satisfying $|H(d)[2]| \gg |H(d)|^{\epsilon}$.  Weaker bounds are known for the size of the
three torsion part; in
\cite{ellenberg-venkatesh-class-group-torsion-bounds} Venkatesh and
Ellenberg (improving on Helfgott and Venkatesh
\cite{helfgott-venkatesh-class-group-three-torsion} and 
Pierce \cite{pierce-three-torsion-bound}) 
show that 
$$
|H(d)[3]| \ll |d|^{1/3+\epsilon}.
$$
From this and the (GRH-conditional) lower bound $d^{1/2} \ll |H(d)|$, one sees that, for any $\epsilon > 0$ there are only finitely many imaginary quadratic class groups $H(d)$ satisfying $|H(d)[3]| \gg |H(d)|^{2/3-\epsilon}$.

The problem of realizing a given abelian group as an imaginary quadratic class group may be viewed in the context of the following broader questions.
\begin{question}
Given a finite abelian group $G$, does there exist a number field $K$ for which the ideal class group of $K$ is isomorphic to $G$?
\end{question}
The answer to this problem is believed to be yes (one ought to be able to take $K$ to be a real quadratic extension of $\mbq$) but the problem is open in general, in spite of various partial results.  G. Cornell \cite{cornell} proved that every finite abelian group occurs as a
subgroup of the ideal class group of some cyclotomic
field, and Y. Yamamoto \cite{yamamoto} proved that, for any $n \geq 1$, there are infinitely many imaginary quadratic fields whose class group contains $(\mbz/n\mbz)^2$ as a subgroup.  We note that Ozaki \cite{ozaki-p-groups} has shown that any (possibly
non-abelian) $p$-group occurs as the maximal unramified $p$-extension
of some number field $F$. 

Further broadening our perspective, we may also ask:
\begin{question}
Given an abelian group $G$, does there exist a Dedekind domain $D$ for which the ideal class group of $D$ is isomorphic to $G$?
\end{question}
In \cite{claborn}, Claborn answered this question in the affirmative; Leedham-Green subsequently showed that the Dedekind domain $D$ can be taken to be a quadratic extension of a
principal ideal ring.

Finally, we remark that the Cohen-Lenstra heuristics apply to a broader class
of situations where finite abelian groups arise as co-kernels of
random sub-lattices of $\mbz^n$.  For instance, \cite{davidsmith}
contains average results on the group of $\mbz/p\mbz$-rational points
of an elliptic curve which are consistent with the Cohen-Lenstra
heuristics (of course the rank can be at most two in this setting),
and (in much the same spirit as our present consideration of missing
class groups) \cite{bankspappalardishparlinski} considers the question
of which finite rank $2$ abelian groups occur as the group of
$\mbz/p\mbz$-rational points of some elliptic curve $E$ over
$\mbz/p\mbz$.

\subsection{Outline of the paper}
\label{sec:outline-paper}

The organization is as follows:  Section \ref{preliminaries} covers
the preliminary material on Cohen-Lenstra heuristics and the
distribution of $L(1,\chi_d)$.  In Section
\ref{tweakedsoundpropproofsection}, we prove Theorem
\ref{tweakedsoundprop}. In Section \ref{heuristics}, we develop
heuristics which lead to Conjectures \ref{refinedsoundconjecture} and
\ref{mainconjecture}.  In Section \ref{partitions}, we discuss
partition generating functions and give a proof of Theorem
\ref{camstheorem}. In Section \ref{numerics}, we sketch the
techniques used to obtain the numerical evidence.

\subsection{Acknowledgements}
\label{sec:thanks}

{P.K. was partially supported by grants from the G\"oran Gustafsson
  Foundation for Research in Natural Sciences and Medicine, and the
  Swedish Research Council (621-2011-5498).}
{S.H. was partially supported by a grant from the
  Swedish Research Council (621-2011-5498).}
{K.P. was partially supported by Simons Foundation grant \# 209266.}
The computations were performed on resources provided by the Swedish National Infrastructure for Computing (SNIC) at PDC Centre for High Performance Computing (PDC-HPC).
We thank the PDC support (Magnus Helmersson,
  Jonathan Vincent,
  Radovan Bast,
  Peter Gille,
  Jin Gong,
  Mattias Claesson) for their assistance 
concerning technical and implementational aspects 
in making the code run on the PDC-HPC resources.

We would like to thank Andrew Booker, Pete Clark, Henri Cohen, Noam
Elkies, Farshid Hajir, Hendrik W. Lenstra, Steve Lester and Peter
Sarnak for enlightening conversations on the topic.



\section{Preliminaries} \label{preliminaries}

In this section, we briefly review relevant background material.  

\subsection{Cohen-Lenstra heuristics}

When a finite abelian $p$-group $G$ occurs in nature, its likelihood
of occurrence is often found to be proportional to $1/| \aut(G)|$.
This suggests constructing a discrete probability measure $\mu$ on 
\[
\mf{G}_p := \{ \text{isomorphism classes of abelian $p$-groups} \}
\]
by setting
$
\ds \mu(\{G\}) := \frac{c}{| \aut(G) |}
$
for an appropriate positive constant $c$, if possible.  The following
lemma shows that this indeed the case, and is also
useful for evaluating $c$. 
\begin{lemma} \label{sumofclprobslemma}
We have that
\[
\begin{split}
\sum_{G \in \mf{G}_p \atop |G| = p^n} \frac{1}{| \aut(G) |} &= \frac{1}{p^n} \prod_{i = 1}^n \left( 1 - \frac{1}{p^i} \right)^{-1}, \\
\sum_{G \in \mf{G}_p \atop |G| \leq p^n} \frac{1}{| \aut(G) |} &= \prod_{i = 1}^n \left( 1 - \frac{1}{p^i} \right)^{-1}.
\end{split}
\]
\end{lemma}
\begin{proof}
The first equation is \cite[Cor 3.8, p. 40]{cohenlenstra}; the second follows from the first by induction on $n$.
\end{proof}

Let us set
\begin{equation} \label{defofetainftyofp}
\eta_\infty(p) := \prod_{i=1}^\infty \left( 1 - \frac{1}{p^i} \right).
\end{equation}
By taking $n \longrightarrow \infty$ in Lemma \ref{sumofclprobslemma},
we see that one must take $c = \eta_\infty(p)$ in order for
$\mu(\mf{G}_p) = 1$.  In the Cohen-Lenstra model, the
probability of $G$ occurring as the $p$-part of a class group is thus given by
\begin{equation} \label{cohenlenstrameasure}
\mu(\{G\}) := \frac{\eta_\infty(p)}{| \aut(G) |}.
\end{equation}
Lemma \ref{sumofclprobslemma} also has the following useful corollary.  Here and later in the paper, we will also make use of the notation
\[
\begin{split}
\mf{D} &:= \{ \text{negative fundamental discriminants} \}, \\
\mf{D}(x) &:= \{ d \in \mf{D} : \; -d \leq x \}, \\
\mf{D}' &:= \{ q \in \mf{D} : \; -q \text{ is prime} \},\\
\mf{D}'(x) &:= \{ q \in \mf{D}' : \; -q \leq x \}.
\end{split}
\]
Recall that by genus theory, we have 
\[
h(d) \text{ is odd} \; \Longleftrightarrow \; -d \text{ is prime}
\]
for $d \in \mf{D}$ with $d < -8$.  This observation explains the following notation, wherein $P$ denotes any property of positive odd integers.
\begin{equation} \label{defofconditionalprob}
\prob(\text{$h$ satisfies $P$} : \; \text{$h$ is an odd class number} ) := \lim_{x \rightarrow \infty} \frac{\# \{ q \in \mf{D}'(x) : \; \text{$h(q)$ satisfies $P$}\}}{\# \mf{D}'(x) } .
\end{equation}
\begin{cor} \label{clcorollary}
Assuming the Cohen-Lenstra heuristics, for any $n \geq 0$ we have
\begin{equation*} 
\begin{split}
\prob(p^n \nmid h : \; \text{$h$ is an odd class number} ) &= \prod_{i=n}^{\infty} \left( 1 - \frac{1}{p^i} \right) \\
\prob(p^n \parallel h : \; \text{$h$ is an odd class number}) &= \frac{1}{p^n} \prod_{i=n+1}^\infty \left( 1 - \frac{1}{p^i} \right).
\end{split}
\end{equation*}
\end{cor}
\begin{proof}
The Cohen-Lenstra heuristics specify that 
\[
\prob(p^n \nmid h : \; \text{$h$ is an odd class number} ) = \mu( \{ G \in \mf{G}_p : |G| \leq p^{n-1} \} ).  
\]
Together with Lemma \ref{sumofclprobslemma}, this gives the first equation, and the second equation follows from the first since $\prob( p^n \parallel h : h \text{ is an odd class number})$ is equal to
\[
\prob( p^n \mid h : h \text{ is an odd class number}) - \prob( p^{n+1} \mid h : h \text{ is an odd class number}). \qedhere
\]
\end{proof}

\subsection{The class number formula and special values of $L$-functions}

Recall the class number formula, which in our context reads
\begin{equation} \label{CNF}
L(1, \chi_d) = \frac{\pi h(d)}{\sqrt{|d|}} \quad\quad (d \in \mf{D}, d < -8),
\end{equation}
where $L(s, \chi_d) = \sum_{n = 1}^\infty \chi_d(n) n^{-s}$ is the $L$-function attached to the Kronecker symbol $\chi_d := \left( \frac{d}{\cdot} \right)$.  This formula connects the statistical study of class numbers to that of the special values $L(1,\chi_d)$.  Building upon ideas that go back to P.D.T.A. Elliot, A. Granville and K. Soundararajan \cite{granvillesound} proved that, on average over $d \in \mf{D}$, $L(1,\chi_d)$ behaves like a random Euler product.  More precisely, if $\mathbb{X}(p)$ denotes the random variable defined by 
\[
\mathbb{X}(p) := 
\begin{cases}
1 & \text{ with probability } \frac{p}{2(p+1)} \\
0 & \text{ with probability } \frac{1}{p+1} \\
-1 & \text{ with probability } \frac{p}{2(p+1)},
\end{cases}
\]
and $L(1,\mathbb{X})$ denotes the random Euler product
\[
L(1,\mathbb{X}) := \prod_{p} \left( 1 - \frac{\mathbb{X}(p)}{p} \right)^{-1},
\]
then \cite[Theorem 2]{granvillesound} (see also \cite[p. 4]{sound}) implies that, for $|z| \leq \log x / (500(\log \log x)^2)$ and $\Re(z) > -1$, we have
\begin{equation} \label{granvillesoundavgresult}
\sum_{d \in \mf{D}(x)} L(1,\chi_d)^z = | \mf{D}(x) | \cdot
\mathbb{E}(L(1,\mathbb{X})^z) + O\left(| \mf{D}(x) | \exp\left(
    -\frac{\log x}{5 \log \log x} \right) \right), 
\end{equation}
where $\mathbb{E}$ denotes the expected value.  
This leads to the average result
\begin{equation} \label{soundsavgresult}
\sum_{h \leq H} \mc{F}(h) = \frac{3\zeta(2)}{\zeta(3)} H^2 + O\left( H^2 (\log H)^{-1/2 + \ve} \right),
\end{equation}
for any $\ve > 0$ (see \cite[Theorem 1]{sound}).  In the interest of establishing the appropriate constant in Conjecture \ref{refinedsoundconjecture}, we will next prove Theorem \ref{tweakedsoundprop}, which is an analogue of \eqref{soundsavgresult} averaged over \emph{odd} values of $h$.  

\section{The average of $\mc{F}(h)$ over odd values of $h$} \label{tweakedsoundpropproofsection}

In this section we prove Theorem \ref{tweakedsoundprop}, that is we
develop an asymptotic formula for $\ds \sum_{h \leq H \atop h \text{
    odd}} \mc{F}(h)$.  By genus theory, the restriction for $h \geq 3$
to be odd is equivalent to the condition that the associated discriminant
$d$ be \emph{prime}.  As
an auxiliary result, we begin by proving the analogue of
\eqref{granvillesoundavgresult} over prime discriminants. 

\subsection{The distribution of $L(1,\chi)$ over prime discriminants}

We now prove an asymptotic formula for the general moment of $L(1,\chi_{q})$ averaged over $q \in \mf{D}'(x)$.  Our proof generally follows the methods used in \cite[Theorem 2]{granvillesound}, but
the restriction to prime discriminants demands that we use a different probabilistic model than the model $\mathbb{X}$ introduced earlier.  Indeed, $\prob(\mathbb{X}(p) = 0) = 1/(p+1)$ corresponds to the probability that a random fundamental discriminant $d \in \mf{D}$ is divisible by the prime $p$, and one computes
\[
\prob( p \mid d : d \in \mf{D} ) = \frac{| p \mbz/p^2\mbz - \{ 0 \} |}{| \mbz/p^2\mbz - \{ 0 \} |} = \frac{1}{p+1}.
\]
On the other hand, the event $p \mid q$ can happen at most once for $q \in \mf{D}'$, and so we replace $\mathbb{X}$ with $\mathbb{Y}$, where we recall that
\begin{equation} \label{defofmby}
\mathbb{Y}(p) :=
\begin{cases}
1 & \text{ with probability } 1/2 \\
-1 & \text{ with probability } 1/2.
\end{cases}
\end{equation}
The corresponding random Euler product is then
\[
L(1,\mby) := \prod_{p} \left( 1 - \frac{\mby(p)}{p} \right)^{-1}.
\]
We will also make use of the following estimate for the remainder term in the Chebotarev density theorem for quadratic fields.
\begin{proposition} \label{errorinchebprop}
Assume the Generalized Riemann Hypothesis for Dedekind Zeta functions of quadratic number fields.  Then for $d \in \mbn$ and any real non-principal Dirichlet character $\chi$ modulo $d$, we have
\[
\sum_{p \leq x \atop \chi(p) = 1} 1 = \frac{1}{2} \Li(x) + O(x^{1/2} \log dx),
\]
with an absolute implied constant.
\end{proposition}
\begin{proof}
This is a special case of a theorem of Lagarias-Odlyzko on the error
term in the Chebotarev density theorem for general number fields; see
\cite[Theorem 1.3]{lagariasodlyzko} and \cite[Th\'{e}or\`{e}me
2]{serre}.
\end{proof}
As an immediate corollary, one deduces the following analogue of the Polya-Vinogradov Theorem, which gives square-root cancellation of characters sums over \emph{prime} values.
\begin{cor} \label{helpfulcorollary}
Assume the Generalized Riemann Hypothesis for Dedekind Zeta functions of quadratic number fields.  Then for $n \in \mbn$ which is not a square, we have
\[
\left| \sum_{q \in \mf{D}'(x)} \chi_{q}(n) \right| \ll x^{1/2} \log(nx),
\]
with an absolute implied constant.
\end{cor}

The next theorem follows from Corollary \ref{helpfulcorollary}, together with some technical lemmas from \cite{granvillesound}.  In particular, its proof will utilize  several properties of the $z$-th divisor function $d_z(n)$ for $z \in \mbc$, which is characterized by the equation
\[
\zeta(s)^z = \sum_{n=1}^\infty \frac{d_z(n)}{n^s} \quad\quad \left( \Re(s) > 1 \right).
\]
Further note that $d_z(n)$ is a multiplicative function, and for prime powers $n = p^a$ we have that
\begin{equation} \label{dzonprimepowers}
d_z(p^a) = \frac{\Gamma(z+a)}{a! \Gamma(z)} = \frac{z(z+1)(z+2) \dots (z+a-1)}{a!}
\end{equation}

\begin{Theorem} \label{tweakedgranvillesoundprop}
Assume the Generalized Riemann Hypothesis and let $\ve > 0$.  Then, uniformly for $|z| \leq \log x / (500(\log \log x)^2)$, we have
\begin{equation*} \label{granvillesoundassumed}
\sum_{{\begin{substack} {q \in \mf{D}'(x)} \end{substack}}}
L(1,\chi_{q})^z = | \mf{D}'(x) | \cdot \mathbb{E}(L(1,\mathbb{Y})^z) +
O_{\ve} \left( x^{1/2 + \ve} \right). 
\end{equation*}
\end{Theorem}
\begin{proof}
By Lemma 2.3 of \cite{granvillesound}, for any $Z \in \mbr$ with $Z \geq \exp \left( (\log q)^{10}) \right)$ we have
\[
L(1,\chi_{q}^z) = \sum_{n=1}^\infty \chi_{q}(n) \frac{d_z(n)}{n} e^{-n/Z} + O\left( \frac{1}{q} \right).
\]
(Note that, since we are assuming GRH, we may ignore any possible exceptional discriminants.)  Thus we have
\begin{equation} \label{decompofLseries}
\sum_{q \in \mf{D}'(x)} L(1,\chi_{q})^z = \sum_{n = 1}^\infty \frac{d_z(n)}{n} e^{-n/Z} \sum_{q \in \mf{D}'(x)} \chi_{q}(n) + O( \log \log x ).
\end{equation}
The main term in our asymptotic comes from the subsequence $n = m^2$; the other values of $n$ contribute to the remainder term.  Indeed, for $n = m^2$, we have
\[
\sum_{q \in \mf{D}'(x)} \chi_{q}(m^2) = | \mf{D}'(x) | + O(\om(m)),
\]
and the contribution of these terms to \eqref{decompofLseries} is thus
\[
| \mf{D}'(x) | \sum_{m = 1}^\infty \frac{d_z(m^2)}{m^2} e^{-m^2/Z} + O\left( \log \log x + \sum_{m = 1}^\infty \frac{| d_z(m^2) \om(m) |}{m^2} e^{-m^2/Z} \right).
\] 
Using $\om(m) \leq d(m)$ together with the bounds
\[
\sum_{m = 1}^\infty \frac{d_z(m^2) d(m)}{m^2} e^{-m^2/Z} \ll \log(|z| + 2)^{4|z|+4} \ll_\ve x^\ve
\]
and
\[
\sum_{m=1}^\infty \frac{d_z(m^2)}{m^2} \left( 1 - e^{-m^2/Z} \right) \leq \sum_{m=1}^\infty \frac{d_{(|z|+1)^2}(m)}{m^2} \left( \frac{m^2}{Z} \right)^{1/4} = \frac{\zeta(3/2)^{(1+|z|)^2}}{Z^{1/4}} \leq \frac{1}{x}
\]
(see \cite[p. 1014]{granvillesound}), one finds that the contribution of the
$n=m^2$ terms to \eqref{decompofLseries} is thus
\[
\begin{split}
| \mf{D}'(x) | \sum_{m=1}^\infty \frac{d_z(m^2)}{m^2} + O_\ve(x^\ve) 
&=
| \mf{D}'(x) | \prod_{p} \left( \sum_{j = 0}^\infty \frac{d_z(p^{2j})}{p^{2j}} \right) + O_\ve(x^\ve) \\
&=
| \mf{D}'(x) | \prod_{p} \left( \sum_{j = 0}^\infty \binom{-z}{2j} \frac{1}{p^{2j}} \right) + O_\ve(x^\ve) \\
&=
| \mf{D}'(x) | \prod_{p} \frac{1}{2} \left( \left(1+\frac{1}{p}\right)^{-z} + \left(1-\frac{1}{p}\right)^{-z} \right) + O_\ve(x^\ve) \\
&=
| \mf{D}'(x) | \cdot \mbe(L(1,\mby)^z) + O_\ve(x^\ve),
\end{split}
\]
where we have used \eqref{dzonprimepowers} together with the binomial series expansions of $\ds \left(1+\frac{1}{p}\right)^{-z}$ and $\ds \left(1-\frac{1}{p}\right)^{-z}$.  In order to handle the terms $n \neq \Box$, we begin by inserting the result of Corollary \ref{helpfulcorollary} into the right-hand side of \eqref{decompofLseries}, obtaining
\begin{equation} \label{thenonsquareterms}
\begin{split}
\left| \sum_{n = 1 \atop n \neq \Box}^\infty \frac{d_z(n)}{n} e^{-n/Z} \sum_{q \in \mf{D}'(x)} \chi_{q}(n) \right| &\ll x^{1/2} \log x\sum_{n = 1}^\infty \frac{| d_z(n) |}{n} e^{-n/Z} \log n \\
&\ll x^{1/2} \log x \sum_{n = 1}^\infty \frac{ d_{\lceil |z| \rceil}(n)}{n} e^{-n/Z} \log n,
\end{split}
\end{equation}
where we have used $| d_z(n) | \leq d_{|z|}(n)$ and
$d_{t_1}(n) \leq d_{t_2}(n)$ for $t_1, t_2 \in \mbr_{>0}$ and
$t_1 \leq t_2$.  In \cite[(2.4), p. 1001]{granvillesound} it is observed that
$\sum_{n=1}^\infty \frac{d_k(n)}{n} e^{-n/Z} \leq (\log 3Z)^k$ for any
positive integer $k$ and real number $Z \geq 2$.  One may adapt that
argument to obtain a similar bound for
$\sum_{n=1}^\infty \frac{d_k(n)}{n} e^{-n/Z} \log n$ by introducing
the function
\[
\widetilde{\log}(t) :=
\begin{cases}
2 & \text{ if } t < e^2 \\
\log t & \text{ if } t \geq e^2.
\end{cases}
\]
Note that, for any $a_1, a_2, \dots, a_k \in \mbn$ we have
\[
\log(a_1 \cdot a_2 \cdot \dots \cdot a_k) \leq \widetilde{\log}(a_1 \cdot a_2 \cdot \dots \cdot a_k) \leq \widetilde{\log}(a_1) \cdot  \widetilde{\log} (a_2)  \cdot \dots \cdot  \widetilde{\log} (a_k).
\]
Furthermore, by estimating a discrete sum by a continuous integral we
may see that, for $Z$ large enough, 
\[
\sum_{a=1}^\infty \frac{e^{-a/Z}}{a}  \widetilde{\log}(a) \ll (\log (e^2 \cdot Z))^2.
\]
Using these facts together with the inequality $\ds d_k(n) e^{-n/Z} \leq e^{k/Z} \sum_{a_1 a_2 \dots a_k = n} e^{-(a_1 + a_2 + \dots + a_k)/Z}$, we find that
\[
\sum_{n=1}^\infty \frac{d_k(n)}{n} e^{-n/Z} \log n \leq \left( e^{1/Z}
  \sum_{a=1}^\infty \frac{e^{-a/Z}}{a}  \widetilde{\log}(a) \right)^k
\leq (\log( e^2 \cdot Z))^{3k}, 
\]
for $Z$ large enough.  Inserting this into \eqref{thenonsquareterms} and taking $Z = \exp\left( (\log x)^{10} \right)$,
we obtain
\[
\begin{split}
\left| \sum_{n = 1 \atop n \neq \Box}^\infty \frac{d_z(n)}{n} e^{-n/Z} \sum_{q \in \mf{D}'(x)} \chi_{q}(n) \right|
&\ll x^{1/2} \log x (\log(e^2 \cdot Z))^{3 \lceil |z| \rceil}, \\
&\ll_\ve x^{1/2 + \ve}.
\end{split}
\]
This completes the proof of Theorem \ref{tweakedgranvillesoundprop}.
\end{proof}

\subsection{The proof of Theorem \ref{tweakedsoundprop}}

We will largely follow the proof of \cite[Theorem 1]{sound} with critical  modifications in appropriate places; we include the details here for completeness.  We make use of the smooth cut-off function
\[
\mf{H}_{c,\gd}(x) := \frac{1}{2\pi i} \int_{c-i\infty}^{c+i\infty} \frac{x^s}{s} \left( \frac{(1+\gd)^{s+1} - 1}{\gd(s+1)} \right) ds,
\]
where the parameters $c, \gd > 0$ will be specified soon.  For any $c, \gd > 0$ we have
\begin{equation} \label{Hascutoff}
\mf{H}_{c,\gd}(x) =
\begin{cases}
1 & \text{ if } x \geq 1 \\
(1 + \gd - 1/x)/\gd & \text{ if } (1+\gd)^{-1} \leq x \leq 1 \\
0 & \text{ if } x \leq (1 + \gd)^{-1}.
\end{cases}
\end{equation}
Just as in \cite{sound}, by using \cite[Theorem 4]{granvillesound}
, one obtains that
\begin{equation} \label{cutoffatX}
\sum_{h \leq H \atop h \text{ odd}} \mc{F}(h) = \sum_{q \in \mf{D}'(X) \atop h_q \leq H} 1 + O_A\left( \frac{H^2}{(\log H)^A} \right)
\end{equation}
for any $A > 0$, where $X := H^2 \log \log H$.  By the class number formula, \eqref{cutoffatX} and \eqref{Hascutoff}, it follows that
\[
\sum_{h \leq H \atop h \text{ odd}} \mc{F}(h) \leq \sum_{q \in \mf{D}'(X)} \mf{H}_{c,\gd}\left( \frac{\pi H}{\sqrt{q} L(1,\chi_q)} \right) + O_A\left( \frac{H^2}{(\log H)^A} \right)\leq \sum_{h \leq H(1+\gd) \atop h \text{ odd}} \mc{F}(h).
\]
We will now work with the main term in the middle above, which is
\begin{equation} \label{whatwefocuson}
\frac{1}{2\pi i} \int_{c-i\infty}^{c+i\infty} \sum_{q \in \mf{D}'(X)} \left( \frac{\pi}{\sqrt{q}L(1,\chi_q)} \right)^s \frac{H^s}{s} \left( \frac{(1+\gd)^{s+1} - 1}{\gd(s+1)} \right) ds.
\end{equation}
We will put $c := 1/\log H$ and $\gd := 1/(\log H)^{1/2}$.  We furthermore set $S := \log X / (10^4 (\log \log X)^2)$ and decompose the above interval into
\[
\int_{|s| \leq S} \; + \; \int_{|s| > S}.
\]
The second term is easily seen to be
\[
\ll \frac{\mf{D}'(X)}{\gd} \int_{|s| > S} \frac{1}{|s(s+1)|} |ds| \ll \frac{H^2}{(\log H)^{3/2 - \ve}}.
\]
For the integral over $|s| \leq S$, we will use Theorem \ref{tweakedgranvillesoundprop} to re-write the integrand in terms of the appropriate moment of $L(1,\mby)$ and then reinterpret $\mf{H}_{c,\gd}$ as a smooth cut-off function as in \eqref{Hascutoff}.  First note that the following equation follows immediately from Theorem \ref{tweakedgranvillesoundprop} by partial summation:
\[
\sum_{q \in \mf{D}'(X)} \left( \sqrt{q} L(1,\chi_q) \right)^{-s} = \mbe(L(1,\mby)^{-s}) \int_1^X t^{-s/2} d\mf{D}'(t) + O_\ve(X^{1/2+\ve}).
\]
Thus, \eqref{whatwefocuson} is equal to
\begin{equation} \label{newtofocuson}
\begin{split}
&\frac{1}{2\pi i} \int_{|s| \leq S} \mbe(L(1,\mby)^{-s}) \int_1^X t^{-s/2} d \mf{D}'(t) \frac{(\pi H)^s}{s} \left( \frac{(1+\gd)^{s+1} - 1}{\gd(s+1)} \right) ds + O_\ve \left( \frac{H^2}{(\log H)^{3/2 - \ve}} \right) \\
= & \mbe\left( \int_1^X \frac{1}{2\pi i} \int_{|s| \leq S} \left( \frac{\pi H}{\sqrt{t} L(1,\mby)} \right)^s \frac{1}{s} \left( \frac{(1+\gd)^{s+1} - 1}{\gd(s+1)} \right) ds \, d \mf{D}'(t) \right) + O_\ve \left( \frac{H^2}{(\log H)^{3/2 - \ve}} \right)
\end{split}
\end{equation}
Extending the integral to $\int_{c-i\infty}^{c+i\infty}$ and managing
the error, we find that  
\[
\frac{1}{2\pi i} \int_{|s| \leq S} \left( \frac{\pi H}{\sqrt{t} L(1,\mby)} \right)^s \frac{1}{s} \left( \frac{(1+\gd)^{s+1} - 1}{\gd(s+1)} \right) ds = \mf{H}_{c,\gd}\left( \frac{\pi H}{\sqrt{t} L(1,\mby)} \right) + O_\ve \left( \frac{L(1,\mby)^{-c}}{(\log H)^{3/2 - \ve}} \right).
\]
Inserting this into \eqref{newtofocuson}, we find that \eqref{whatwefocuson} is equal to 
\begin{equation} \label{newnewtofocuson}
\begin{split}
&\mbe \left( \int_1^{\min \left( \frac{\pi^2 H^2}{L(1,\mby)^2}, X \right)} d \mf{D}'(t) + O_\ve \left( \frac{H^2}{(\log H)^{3/2 - \ve}}(1 + L(1,\mby)^{-c}) \right) \right) \\
= & \frac{1}{2} \mbe \left( \Li \left( \min \left( \frac{\pi^2 H^2}{L(1,\mby)^2}, X \right) \right) \right) + O_\ve \left( \frac{H^2}{(\log H)^{3/2 - \ve}} \right).
\end{split}
\end{equation}
Now using \cite[Proposition 1]{granvillesound}, we find that $\min \left( \frac{\pi^2 H^2}{L(1,\mby)^2}, X \right) = \frac{\pi^2 H^2}{L(1,\mby)^2} + O_A\left( \frac{H^2}{(\log H)^A} \right)$ for any $A > 0$, and so we find that \eqref{newnewtofocuson} becomes
\[
 \frac{1}{2} \mbe \left( \Li \left( \frac{\pi^2 H^2}{L(1,\mby)^2} \right) \right) + O_\ve \left( \frac{H^2}{(\log H)^{3/2 - \ve}} \right).
\]
Finally, using the asymptotic $\ds \Li(x) \sim \frac{x}{\log x}$ together with the calculation
\begin{gather}
\begin{split}
\mbe(L(1,\mby)^{-2}) = \prod_{p} \mbe \left( \left( 1 - \frac{\mby(p)}{p} \right)^2 \right) &= \prod_{p} \left( \frac{1}{2} \left( 1 - \frac{1}{p} \right)^2 + \frac{1}{2} \left( 1 + \frac{1}{p}\right)^2 \right) \\
&= \prod_{p} \left( 1 - \frac{1}{p^4} \right) \left( 1 - \frac{1}{p^2} \right)^{-1} = \frac{\zeta(2)}{\zeta(4)} = \frac{15}{\pi^2},
\end{split} \label{evalofc0}
\end{gather}
the proof of Theorem \ref{tweakedsoundprop} is concluded.
\begin{remark} \label{moreaccurateremark}
Our proof shows that in fact
\[
\sum_{h \leq H \atop h \text{ odd}} \mc{F}(h) = \frac{1}{2} \mbe \left( \Li \left( \frac{\pi^2 H^2}{L(1,\mby)^2} \right) \right) + O_\ve  \left( \frac{H^2}{(\log H)^{3/2 - \ve}} \right).
\]
We find that the main term in the above expression fits the numerical
data much better than the asymptotically equivalent formula given in
Theorem \ref{tweakedsoundprop}, though it must be stressed that the
corrections are of lower order than the error term.  In the tables
presented in Sections 
\ref{introduction} and \ref{numerics}, the number listed under
``predicted'' refers to the higher order expansion of
$\ds \frac{\mf{C}}{15} \cdot \mf{c}(h) \cdot h \cdot \mbe \left(
  \frac{1}{L(1,\mby)^2 \log (\pi h / L(1,\mby))} \right)$
given in \eqref{realpred}.
\end{remark}

\section{Heuristics} \label{heuristics}

\subsection{Heuristics for Conjecture \ref{refinedsoundconjecture}}
\label{sec:sound-conj}

Recall from Remark \ref{moreaccurateremark} that we have
\begin{gather*}
  \sum_{\substack{ h\leq H \\ h\text{ odd} }} \mathcal F(h) \approx \frac12\E{ \text{Li}\left(
    \dfrac{\pi^2 H^2}{L(1,\mathbb Y)^2}
  \right)}.
\end{gather*}
Denote the right-hand side by $G(H)$.
An average order of $\mathcal F$ is given by 
\begin{gather*}
  G(h)-G(h-2) \approx 
  2 \frac d{dh} G'(h) =
  \E{ \frac d{dh} \text{Li}\left(
    \dfrac{\pi^2 h^2}{L(1,\mathbb Y)^2}
  \right)} = \\
  2\pi^2 h \E{ 
    \dfrac{1}{
      L(1,\mathbb Y)^2
      \log\left( {\pi^2 h^2}/{L(1,\mathbb Y)^2}\right)
    }
  } = 
  \frac{\pi^2 h}{\log(\pi h)} \E{ 
    \dfrac1{L(1,\mathbb Y)^2}
    \dfrac{1}{
      1-\frac{\log L(1,\mathbb Y)}{\log(\pi h)}
    }
  }.
\end{gather*}
With a high probability, we have $\log L(1,\mathbb Y)/\log (\pi h)< 1$ for large $h$, so the above can be approximated with
\begin{gather}
  \frac{\pi^2 h}{\log(\pi h)} \E{ 
    \dfrac1{L(1,\mathbb Y)^2}
    \left(
      1
      +\frac{\log L(1,\mathbb Y)}{\log(\pi h)}
      +\frac{\log^2 L(1,\mathbb Y)}{\log^2(\pi h)}
      +\cdots
    \right)
  }. \label{infterms}
\end{gather}
We will approximate this by keeping the first few terms in the innermost
parentheses. In this regard, define 
$c_0\defeq \E{\frac1{L(1,\mathbb Y)^2}}$, $c_1 \defeq
\frac1{c_0}\E{\frac{\log L(1,\mathbb Y)}{L(1,\mathbb Y)^2}}$, $c_2
\defeq \frac1{c_0}\E{\frac{\log^2 L(1,\mathbb Y)}{L(1,\mathbb
    Y)^2}}$, and $c_3\defeq \frac1{c_0}\E{\frac{\log^3 L(1,\mathbb
    Y)}{L(1,\mathbb Y)^2}}$. Recall from \eqref{evalofc0} that
$c_0=15/\pi^2$. The 
constants $c_1, c_2$ and $c_3$ may be calculated to arbitrary precision as
follows. 
Write $L_p\defeq 1-\frac{\mathbb Y(p)}p$. Then $L(1,\mathbb Y)=\prod_p L_p\inv$ and $\log L(1,\mathbb Y)=-\sum_p\log L_p$. Now
\begin{gather}
  \E{\frac{\log L(1,\mathbb Y)}{L(1,\mathbb Y)^2}} = \E{-\sum_p \log L_p\prod_r L_r^2}
  = -\sum_p\E{L_p^2 \log L_p } \prod_{r\neq p} \E{L_r^2} 
  = -c_0 \sum_p\frac{ \E{L_p^2 \log L_p } }{\E{L_p^2}} \label{evalofc1}
\end{gather}
where $\E{L_p^2}=1+\frac1{p^2}$ and $\E{L_p^2\log L_p}=\frac12\left((1-\frac1p)^2\log(1-\frac1p)+(1+\frac1p)^2\log(1+\frac1p)\right)$. Next
\begin{gather}
  \E{\frac{\log^2 L(1,\mathbb Y)}{L(1,\mathbb Y)^2}} =
  \E{
    \sum_{p,q}\log L_p\log L_q\prod_{r}L_r^2 
  } = \nonumber \\
  \E{
    \sum_{p\neq q}L_p^2L_q^2\log L_p\log L_q\prod_{r\neq p,q}L_r^2 +
    \sum_{p}L_p^2(\log L_p)^2\prod_{r\neq p}L_r^2
  } = \nonumber \\
  \sum_{p\neq q}\E{L_p^2\log L_p}\E{L_q^2\log L_q}\prod_{r\neq p,q}\E{L_r^2} +
    \sum_{p}\E{L_p^2(\log L_p)^2}\prod_{r\neq p}\E{L_r^2} = \nonumber \\
  c_0\sum_{p\neq q}\frac{\E{L_p^2\log L_p}\E{L_q^2\log L_q}}{\E{L_p^2}\E{L_q^2} } +
    c_0\sum_{p}\frac{\E{L_p^2(\log L_p)^2}}{\E{L_p^2}} = \nonumber \\
  c_0\cdot \left(
      \left(\sum_{p}\frac{\E{L_p^2\log L_p}}{\E{L_p^2} }\right)^2 
      - \sum_{p}\left(\frac{\E{L_p  ^2\log L_p}}{\E{L_p^2} }\right)^2
      + \sum_{p}\frac{\E{L_p^2(\log L_p)^2}}{\E{L_p^2}}
    \right) \label{evalofc2}
\end{gather}
where $\E{L_p^2(\log L_p)^2}=\frac12\left((1-\frac1p)^2\log^2(1-\frac1p)+(1+\frac1p)^2\log^2(1+\frac1p)\right)$.
One may similarly show that
\begin{gather}
  -\frac1{c_0}\E{\frac{\log^3 L(1,\mathbb Y)}{L(1,\mathbb Y)^2}} =  
    \sum_{\substack{p,q,r\\ \text{distinct}}}\frac{\E{(\log L_p)L_p^2}}{\E{L_p^2}} \frac{\E{(\log L_q)L_q^2}}{\E{L_q^2}} \frac{\E{(\log L_r)L_r^2}}{\E{L_r^2}} + \nonumber \\
    3\sum_{\substack{p\neq r}}\frac{\E{(\log L_p)^2L_p^2}}{\E{L_p^2}}   \frac{\E{(\log L_r)L_r^2}}{\E{L_r^2}} + 
    \sum_{\substack{p}}\frac{\E{(\log L_p)^3L_p^2}}{\E{L_p^2}}. \label{evalofc3}
\end{gather}
Calculating the expressions \eqref{evalofc1}, \eqref{evalofc2} and \eqref{evalofc3} with
$10^5$ prime terms yields
\begin{gather}
  \begin{split}
  c_1 &\approx -0.578071, \\
  c_2 &\approx +0.604049, \\
  c_3 &\approx -0.526259.
  \end{split} \label{c1andc2}
\end{gather}
Thus, taking the first four terms of \eqref{infterms}, an approximation of $\mathcal F(h)$ for odd $h$ is
\begin{gather}
  \frac{\pi^2 h}{\log(\pi h)}
    \left(
      c_0
      +\frac{c_0c_1}{\log(\pi h)}
      +\frac{c_0c_2}{\log^2(\pi h)}
      +\frac{c_0c_3}{\log^3(\pi h)}
    \right) =
  \frac{15 h}{\log(\pi h)}
    \left(
      1
      +\frac{c_1}{\log(\pi h)}
      +\frac{c_2}{\log^2(\pi h)}
      +\frac{c_3}{\log^3(\pi h)}
    \right). \label{flatprediction}
\end{gather}
However, this assumes that each number $h$ occurs as $h(d)$ with equal frequency, which is inconsistent with Corollary \ref{clcorollary}.  We thus introduce the correction factor
\newcommand{\bigproduct}{\prod_{\substack{p\geq 3 \text{ prime} \\ n\geq 0 \\ p^n \parallel h  }}}
\begin{gather}
\begin{split}
\tilde{\mf{c}}(h) := \bigproduct \frac{\prob(p^n \parallel h' : \; h' \text{ is an odd class number})}{\prob(p^n \parallel h' : \; h' \text{ is an odd integer})} &=
\bigproduct \frac{p^{-n} \prod_{i = n+1}^\infty \left( 1 - \frac{1}{p^i} \right)}{ p^{-(n+1)}(p-1)} \\
&= \bigproduct \left( 1 - \frac{1}{p} \right)^{-1} \prod_{i = n+1}^\infty \left( 1 - \frac{1}{p^i} \right)
\end{split} \label{localfactor}
\end{gather}
In the above, in addition to using \eqref{defofconditionalprob}, we are also using
\[ 
\prob(\text{$h$ satisfies $P$} : \; \text{$h$ is an odd integer}) := \lim_{x \rightarrow \infty} \frac{\#\{\text{$h \in \mbn : h \text{ is odd,} \; h \leq x, \, h$ satisfies $P$}\}}{\#\{\text{$h \in \mbn : \; h \text{ is odd,} \; h \leq x$} \} }.
\]
We emphasize that $n = 0$ is allowed in \eqref{localfactor}, and so the expression defining $\tilde{\mf{c}}(h)$ is an \emph{infinite} product.  Note that, heuristically at least, we have
\begin{equation} \label{averagestoone}
\sum_{h \leq H \atop h \text{ odd}} \tilde{\mf{c}}(h) \sim \frac{H}{2}, \quad \quad (H \longrightarrow \infty).
\end{equation}
Indeed, if
$\ds \sum_{h \leq H \atop h \text{ odd}} \tilde{\mf{c}}(h) \sim B
\cdot \frac{H}{2}$,
then $B$ has expected value
\[
\begin{split}
B &= \prod_{p \text{ odd}} \sum_{n = 0}^\infty \prob(p^n \parallel h : \; h \text{ is an odd integer}) \cdot \frac{\prob(p^n \parallel h : \; h \text{ is an odd class number})}{\prob(p^n \parallel h : \; h \text{ is an odd integer})} \\
&= \prod_{p \text{ odd}} \sum_{n = 0}^\infty \prob(p^n \parallel h : \; h \text{ is an odd class number}) \\
&= 1.
\end{split}
\]
Noting that
\[
\tilde{\mf{c}}(h) = \prod_{\ell = 3 \atop \ell \text{ prime}}^\infty \prod_{i=2}^\infty \left( 1 - \frac{1}{\ell^i} \right) \cdot \mf{c}(h)
= \dfrac{\mathfrak C}{15}\cdot \mathfrak c(h),
\]
we get Conjecture \ref{refinedsoundconjecture} by multiplying the average order \eqref{flatprediction} with the local correction factor \eqref{localfactor}.

\subsection{Dampening the three divisibility bias}
\label{sec:damp-three-divis}
Given an odd natural number $h$, let $k\leq 11$ and $n\leq k-3$ be such that $h\in[3^k,3^{k+1})$ and $3^n\parallel h$.
We define the adjustment $\operatorname{pred}'(h)$ by replacing in $\operatorname{pred}(h)$ the factor
\[\prob(3^n \parallel h' : \; h' \text{ is an odd class number})=3^{-n} \prod_{i = n+1}^\infty \left( 1 - \frac{1}{3^i} \right)\]
in $\tilde {\mf c}(h)$ coming from the Cohen-Lenstra heuristic, by the observed value
\begin{gather}
  \prob(3^n \parallel h' : \; h' \text{ is an odd class number} \in [3^k,3^{k+1})) =
  \left .
    \sum_{\substack{ h'\in[3^k,3^{k+1}) \\ h'\text{ odd}  \\ 3^n \parallel h' }} \mathcal F(h')
  \right /
    \sum_{\substack{ h'\in[3^k,3^{k+1}) \\ h'\text{ odd}  }} \mathcal F(h')
\end{gather}
using our computed values of $\mc F(h)$ (see Section \ref{numerics}).

As mentioned earlier, this three divisibility bias is connected to other recent work:  Belabas \cite{belabas-fast-cubic} noted rather slow convergence in the Davenport-Heilbronn asymptotic average of $H(d)[3]$; Roberts
\cite{roberts-cubic-fields-secondary-term} later conjectured that this
was due to a negative second order term of size $X^{5/6}$ (here the
main term is of order $X$).  Robert's conjecture was recently proved
in
\cite{bhargava-shankar-tsimerman-davenport-heilbronn-secondary,taniguchi-thorne-secondary-term-cubic-fields}.

\subsection{Heuristics for Conjecture \ref{mainconjecture}}

We now give heuristics supporting Conjecture \ref{mainconjecture}.  Let $\gl = (n_1, n_2, \dots, n_r)$ be a partition of $n$, so that
\begin{equation} \label{partitionofn}
n_1 \geq n_2 \geq \dots \geq n_r \geq 1
\end{equation}
and $n_1 + n_2 + \dots + n_r = n$, and let $\ds G_\gl(p) := \bigoplus_{i=1}^r \mbz/p^{n_i}\mbz$ be the corresponding abelian group.  By the assumption \eqref{factorizationofmcFofG}, the expected value of $\mc{F}(G_\gl(p))$ is
\begin{equation} \label{decompatGofp}
\mc{F}(G_\gl(p)) \approx P(G_\gl(p)) \cdot \mc{F}(p^n).
\end{equation}
The following proposition evaluates $P(G_\gl(p))$ explicitly.  Let $k$ be the number of distinct parts of $\gl$, and let $m_1, m_2, \dots, m_k$ be the multiplicity of each distinct part.  Thus, \eqref{partitionofn} reads
\[
n_1 = \dots = n_{m_1} > n_{m_1 +1} = \dots = n_{m_1 + m_2} > \dots > n_{\sum_{i=1}^{k-1}m_i+1} = \dots = n_{\sum_{i=1}^{k}m_i}.
\]
\begin{proposition} \label{mainthm}
With the notation just given, we have
\begin{equation} \label{evaluationofautgnu}
P( G_\gl(p) )  = p^{\cyc(\gl) - n} \cdot \prod_{i = 1}^k \prod_{j = 1}^{m_i} \left( 1 - \frac{1}{p^j} \right)^{-1} \prod_{i=1}^n \left( 1 - \frac{1}{p^i} \right),
\end{equation}
where $\cyc(\gl)$ is given by \eqref{defofcyclicityindex}.  In particular, as $p \longrightarrow \infty$, we have that
\begin{equation} \label{usefulasymptotic}
P(G_\gl(p)) \sim p^{\cyc(\gl)-n}.
\end{equation}
\end{proposition}
\begin{proof}
The statement follows immediately by combining Lemma \ref{sumofclprobslemma} with the formula
\[
| \aut(G_\gl(p)) | = p^{2n-\cyc(\gl)} \prod_{i = 1}^k \prod_{j = 1}^{m_i} \left( 1 - \frac{1}{p^j} \right).
\]
This formula is classical, having appeared in a 1907 paper of A. Ranum \cite{ranum}.  For a more modern exposition, see \cite{hillarrhea} or \cite{lengler2}.
\end{proof}
Inserting \eqref{usefulasymptotic} together with Conjecture \ref{refinedsoundconjecture} into the right-hand side of \eqref{decompatGofp}, and observing that $\mf{c}(p^n) \longrightarrow 1$ as $p \longrightarrow \infty$, we see that Conjecture \ref{mainconjecture} follows.  In the case $\cyc(\gl) = 0$ we write
\[
\sum_{p \leq x} \mc{F}(G_\gl(p)) \sim \sum_{p \leq x} P(G_\gl(p)) \cdot \mc{F}(p^n) \sim  \sum_{p \leq x} \mf{C} \cdot \frac{p^{\cyc(\gl)}}{\log (p^n)}
\]
and use partial summation.

\section{Attainable partitions are very rare} 
\label{partitions}

We now prove Theorem \ref{camstheorem}. 
To this end, let $c_{n,r}$ denote the number of attainable
partitions of $n$ into $r$ parts.  Work of Sellers (\cite{Sel1},\cite{Sel2}) leads to a generating
function for the number of partitions of $n$ which satisfy a certain
type of linear inequality amongst their parts:
\begin{Theorem}[\cite{Sel1},\cite{Sel2}]
The number of partitions $\gl = (n_1,n_2,\ldots, n_r)$ of
$n$ into $r$ non-negative parts satisfying the inequality
$n_1\geq \sum_{i=2}^r b_in_i$, for some non-negative
integers $b_i$ with $b_2>0$, has generating function
\[
\frac{1}{(1-x)(1-x^{b_2+1})(1-x^{b_2+b_3+2})(1-x^{b_2+b_3+b_4+3})\cdots(1-x^{b_2+b_3+\cdots+b_r+r-1})}.
\]
\end{Theorem}


Applying this result to our context requires a slight modification,
and leads to the following generating function for the attainable
partitions.

\begin{cor}\label{GenFun}
The generating function $C_r(x)$ for the sequence $c_{n,r}$ of
length-$r$
attainable partitions of $n$ is given by
\[
C_r(x)=\sum_{n=0}^\infty c_{n,r}x^n=
\frac{x^{r^2-r}}{(1-x)\prod_{j=2}^r(1-x^{j^2-j})}.
\]
\end{cor}

\begin{proof}
First, observe that by definition we require our partitions to be
comprised of positive (rather than non-negative) parts. To accommodate
this change, we use the easily-verified bijection between partitions
of $n$ into $r$ \emph{non-negative} parts satisfying the inequality
$b_1\geq \sum_{i=2}^r b_in_i$ and partitions of $n+\sum_{i=2}^r b_i+
(r-1)$ into $r$ \emph{positive} parts satisfying the same inequality,
given by
\[
(n_1, \ldots, n_r) \to (n_1+\sum_{i=2}^r b_i , n_2+1 , ... , n_r+1)
\]

Thus the analogous generating function to Seller's above for
partitions into positive parts is simply a shift of indices away,
given by 
\[
\frac{x^{b_2+b_3+\cdots+b_r+r-1}}{(1-x)(1-x^{b_2+1})(1-x^{b_2+b_3+2})
(1-x^{b_2+b_3+b_4+3})\cdots(1-x^{b_2+b_3+\cdots+b_r+r-1})}.
\]
Finally, we apply this to attainable
partitions, which by definition satisfy an inequality in the form of
the theorem, with coefficients $b_i=2i-3$. The corollary then follows
from the observation
\[
j-1+\sum_{i=2}^j b_i=j-1+\sum_{i=2}^j (2i-3)=j^2-j
\]
for any $2\leq j\leq r$.

\end{proof}

Basic results about growth rates about coefficients of rational
generating functions leads to an asymptotic count of attainable
partitions:

\begin{cor}\label{Factorial}
For fixed $r$, the proportion of length-$r$ partitions of $n$ which
are attainable is asymptotically $\frac{1}{(r-1)!}$.
\end{cor}

\begin{proof}

We rewrite our expression for $C_r(x)$ to isolate its singularity on
the unit circle with the highest multiplicity ($x=1$ with multiplicity
$r$) and apply the techniques of singularity analysis. Namely, we
write
\[
\sum_{n=0}^\infty c_{n,r}x^n=\frac{1}{(1-x)^r}\cdot\frac{x^{r^2-r+1}}
{\prod_{j=2}^{r}(1+x+x^2+\cdots+x^{j^2-j-1})}=:\frac{f_r(x)}{(1-x)^r},
\]
where here $f_r(x)$ is analytic at $x=1$. A partial fraction
decomposition
shows that the asymptotics for the coefficients are governed by this
singularity (see, e.g., \cite[p. 256]{flajolet-sedgewick-book}), and we obtain
\[
c_{n,r} \sim \frac{f_r(1)n^{r-1}}{(r-1)!}=\frac{n^{r-1}}{(r-1)!\prod_{j=2}^{r}(j^2-j)}=\frac{n^{r-1}}{r!(r-1)!^2}.
\]
Similarly, by the well-known generating function
\[
\sum_{n=0}^\infty p_{n,r}x^n=\frac{x^r}{\prod_{j=1}^r(1-x^j)}=\frac{1}{(1-x)^r}\cdot\frac{x^r}{\prod_{j=2}^r(1+x+x^2+\cdots+x^{j-1})},
\]
for $p_{n,r}$, the total number of length-$r$ partitions of $n$, we conclude that
\[
p_{n,r}\sim \frac{n^{r-1}}{r!(r-1)!}.
\]
Taking the ratio of these gives 
\[
\lim\limits_{n\to\infty}\frac{c_{n,r}}{p_{n,r}}=\lim\limits_{n\to\infty}\frac{\frac{n^{r-1}}{r!(r-1)!^2}}{\frac{n^{r-1}}{r!(r-1)!}}=\frac{1}{(r-1)!},
\]
proving the result.
\end{proof}

Moving from the fixed rank to the fixed order case, we set
\[
c_n=\sum_{r=1}^n c_{n,r},
\]
the total number of partitions of $n$ which are attainable. 

\begin{Lemma}\label{RecRel}
For fixed $n$, the numbers $c_{n,r}$ satisfy the recurrence relation
\[
c_{n,r+1}=\sum_{i=0}^{\lfloor \frac{n-2r}{r^2+r}\rfloor}c_{n-2r-i(r^2+r),r}.
\]
\end{Lemma}
\begin{proof}
From Corollary \ref{GenFun} we easily deduce the recurrence relation between the successive generating functions:
\[
C_{r+1}(x)=\frac{x^{2r}}{1-x^{r^2+r}}\, C_r(x)=(1+x^{r^2+r}+x^{2(r^2+r)}+\cdots)(x^{2r}C_r(x)),
\]
from which the lemma follows by equating coefficients.
\end{proof}
We prove by induction that for fixed $n\geq 1$ we have $c_{n,r}\leq
\frac{n^{r-1}}{(r-1)!^2}$ for all $r$.  This is trivial for $r=1$ since
$c_{n,1}=1$.  For the inductive step, the recurrence relation in Lemma
\ref{RecRel} gives {\small
\[
c_{n,r+1}=\sum_{i=0}^{\lfloor \frac{n-2r}{r^2+r}\rfloor}c_{n-2r-i(r^2+r),r}\leq \frac{1}{(r-1)!^2}\sum_{i=0}^{\lfloor \frac{n-2r}{r^2+r}\rfloor}(n-2r-i(r^2+r))^{r-1}.
\]}
The terms in this sum are positive and decreasing as a function of
$i$, and so we can compare to the integral:
\begin{align*}
\sum_{i=0}^{\lfloor \frac{n-2r}{r^2+r}\rfloor}(n-2r-i(r^2+r))^{r-1}&\leq (n-2r)^{r-1}+\int_0^{\lfloor \frac{n-2r}{r^2+r}\rfloor}(n-2r-i(r^2+r))^{r-1}\,di\\
&=n^{r-1}+\frac{n^r}{r(r^2+r)}-\frac{(n-2r-\lfloor \frac{n-2r}{r^2+r}\rfloor(r^2+r))^{r}}{r(r^2+r)}
\end{align*}
Since the latter term is positive and $r^2+r\leq n$, we can continue
\[
c_{n,r+1}\leq\frac{1}{(r-1)!^2}\sum_{i=0}^{\lfloor \frac{n-2r}{r^2+r}\rfloor}(n-2r-i(r^2+r))^{r-1}\leq \frac{1}{(r-1)!^2}\frac{rn^r+n^r}{r(r^2+r)}=\frac{n^r}{r!^2},
\]
completing the induction.  Now, summing over $r$ gives
\[
c_n=\sum_{r=1}^n c_{n,r}\leq \sum_{r=1}^n \frac{n^{r-1}}{(r-1)!^2}\leq \sum_{r=1}^\infty \frac{n^{r-1}}{(r-1)!^2}=I_0(2\sqrt{n}),
\]
where $I_0(x)$ denotes the 0-th modified Bessel function of the first
kind.  By the asymptotic $I_0(x)\sim \frac{e^x}{\sqrt{2\pi x}}$, we
can compare the formula for $c_n$ with the famous asymptotic of
Hardy-Ramanujan \cite{HR}, $p_n\sim
\frac{e^{\pi\sqrt{2n/3}}}{4n\sqrt{3}}$, for the number of partitions
of $n$.  Taking the ratio of the two gives
\[
\frac{c_n}{p_n} \ll n^{3/4}e^{(2-\sqrt{\frac{2}{3}}\pi)\sqrt{n}},
\]
proving Theorem \ref{camstheorem}.

\begin{remark}
Since $c_{n,1}=1$ for all $n\geq 1$ (and $c_{0,1}=0$), Lemma \ref{RecRel} provides
explicit formulas for the number of attainable partitions of $n$ into a small
number of parts.  For example,
\[
c_{n,2}=\sum_{i=0}^{\lfloor \frac{n-2}{2}\rfloor}c_{n-2-2i,1}=\left\lfloor \frac{n-1}{2}\right\rfloor,
\]
agreeing with the easily-checked fact that the partition $[a,b]$ of
$n$ is attainable if and only if $a>b$.  Less trivially, if we
temporarily adopt the simplifying convention that $\lfloor x\rfloor=0$
for $x<0$, we have
\[
c_{n,3}=\sum_{i=0}^{\lfloor \frac{n-4}{6}\rfloor}c_{n-4-6i,2}=\sum_{i=0}^{\lfloor \frac{n-4}{6}\rfloor} \bigg\lfloor \frac{n-5-6i}{2}\bigg\rfloor.
\]
This leads to Rademacher-type formulas for computing the exact value of $c_n$.
\end{remark}

\section{Computer computations} \label{numerics}

With the aid of a supercomputer and assuming GRH, we have
computed $\mathcal F(h)$ and $\mathcal F(G)$ for all odd $h<10^6$ and
all $p$-groups $G$ of odd size at most $10^6$.
We have made the computed values available online, see the references \cite{data-Fh}, \cite{data-FG}, \cite{data-d-Hd}. In this section we will describe how this computation was accomplished.

As already noted, by genus theory, if $-q<-8$ is a fundamental discriminant, then $h(-q)$ is odd precisely when $q$ is prime.
Corollary 1.3 in \cite{soundbound} states that under GRH,
\begin{gather} \label{soundboundeq}
  h(-q) \geq \frac{\pi}{12e^\gamma}\sqrt q\left(
    \log \log q - \log 2 + \frac 12 + \frac1{\log \log q} + \frac{14 \log\log q}{\log q}
  \right)\inv
\end{gather}
if $-q$ is a fundamental discriminant such that $q\geq 10^{10}$. It is easy to verify that the right-hand side above is monotonic for $q\geq 10^{10}$.
This implies that if $q\geq 1.1881\cdot 10^{15}$ then $h(-q)>10^6$. 
\newcommand{\dinterval}{\mathfrak D'(1.1881\cdot 10^{15})}
Thus it suffices to consider only discriminants in
$\dinterval$.

We use the procedure \verb|quadclassunit0| in the
computer package PARI 2.7.3 to compute the class groups $H(d)$; this procedure guarantees correct results
assuming GRH, cf. \cite[Section~3.4.70]{pari}. In principle, doing this for every $d\in\dinterval$ would suffice, but a number of practical speedups were necessary. 

\subsection{Brief description of the algorithm}

We give a brief but not complete description of our algorithm.
\newcommand{\happrox}{h_{\text{approx}}}
Our computer program iterates over all $d\in\dinterval$ and records for each odd $h<10^6$ and each noncyclic 
$p$-group $G$, how many times a group of order $h$ or a group
isomorphic to $G$ is found, avoiding to compute $h(d)$ or $H(d)$ whenever not necessary (note that if $G$ is a cyclic $p$-group, then the value of $\mathcal F(G)$ can be calculated from the data that we are keeping).

Given a fundamental discriminant $d\in \dinterval$, we begin by
calculating an approximation $\happrox$ of $h(d)$ together with an
explicit error factor $E$, by setting
$\happrox \defeq \frac{\sqrt{\abs d}}\pi e^{\nu(x_1,d)}$ and
$E \defeq e^{\eta(x_1, x_2, d)}$ for suitable $x_1,x_2$ using
Proposition \ref{propnum} below.  If we already at this stage can
prove that $h(d)>10^6$ (that is, if $\happrox/ E > 10^6$), then we
discard $d$.  This cuts down our search space by roughly a factor of
100, as the lower bound \eqref{soundboundeq} is overestimated
by roughly this factor in our case.

Otherwise,
we compute a candidate $h^*$ for $h(d)$ 
using Shank's baby-step/giant-step algorithm (specifically, we find an integer $h^*$ near in value to $\happrox$ such that $g^{h^*}$ is the identity element for up to three different group elements $g\in H(d)$).
We only compute one such candidate, but in practice, this candidate agrees with the true value of $h(d)$ (assuming GRH) with a failure rate of about $1.5\cdot 10^{-7}$ for $d$ in our range.

Next, we try to find the exponent of the group by determining the
smallest divisor $e^*$ of $h^*$ such that $g^{e^*}$ is the identity
element for up to 12 different group elements $g\in H(d)$. We have
that $e^*$ divides the order of the group, and if moreover the error
factor $E$ is small enough such that $h^*$ is the unique multiple of
$e^*$ in the interval $\happrox \cdot [\frac1E,E]$ then we have proven
that $h(d)=h^*$.  In practice, this step in our program only catches
cyclic groups and groups of the form $\Z/m\Z\times \Z/3\Z$ and
$\Z/m\Z\times \Z/5\Z$ ($m\geq 3$), but since the majority of the groups $H(d)$
should be of this form\footnote{We expect the class group to be cyclic
  more than $97.7\%$ of the time, and class groups containing
  $\Z/q\Z \times \Z/q\Z$ for prime $q > 5$ are very rare
  (cf. \cite[p. 56]{cohenlenstra}.)}, our
program stops at this step in $99.7\%$ of all cases it is reached.

If any of the above fails (that is, if $g^{h^*}$ was not the identity
element for some $g$ or if $E$ was not small enough) or if $h(d)$ was
determined to be an odd prime power, then we proceed to compute the
structure of the entire class group $H(d)$ using PARI.



\begin{proposition} \label{propnum}
  Assume GRH.
  Let $d<-8$ be a fundamental discriminant and 
  let $x_1<x_2$ be two integers such that $x_1\geq 1$ and $x_2\geq 10^5$. Then
  \begin{gather}
    \frac1{e^{\eta(x_1,x_2,d)}} \leq \dfrac{h(d)}{\frac{\sqrt{\abs d}}\pi e^{\nu(x_1,d)} } \leq e^{\eta(x_1,x_2,d)} \label{happrox}
  \end{gather}
  where
\begin{gather*}
    \nu(x_1,d) \defeq \sum_{p\leq x_1}-\log\left(1-\frac{ (\frac dp) }p\right), \\
    \eta(x_1,x_2,d) \defeq \dfrac{1.562 \log \abs d+0.655 \log x_2}{\sqrt x_2} +
      \log\log x_2+B +\dfrac{3\log x_2+4}{8\pi\sqrt x_2 } -\sum_{p\leq x_1}\dfrac 1p + \dfrac1{x_1},
  \end{gather*}
  and $\ds B := \lim_{x \rightarrow \infty} \left( \sum_{p \leq x} \frac{1}{p} - \log \log x \right) \approx 0.2614972128\ldots$ is the prime reciprocal constant.
\end{proposition}

\begin{proof}
  Let $\chi$ be the real-valued character $(\frac d\cdot)$ of modulus $\abs d>1$.
  Theorem 9.1 combined with Table 4 in \cite{bach} states that under GRH,
  \begin{gather}
    \abs{\log L(1,\chi)-\log \prod_{p<x_2}\dfrac1{1-\frac{\chi(p)}p}   } \leq \dfrac{1.562 \log \abs d+0.655 \log x_2}{\sqrt x_2} \label{moo1} 
  \end{gather}
  for any $x_2\geq 10^5$. 
  By Taylor expansion, we have
  \begin{gather}
    \begin{split}
    \log \prod_{p<x_2}\dfrac1{1-\frac{\chi(p)}p} 
    &= \sum_{p<x_2}-\log\left(1-\frac{\chi(p)}p\right) \\
    &= \sum_{p\leq x_1}-\log\left(1-\frac{\chi(p)}p\right)+
       \sum_{x_1<p<x_2} \sum_{m=1}^\infty\frac1m\left(\dfrac{\chi(p)}p\right)^m.
   \end{split} \label{moo2}
  \end{gather}
  We can bound the terms with $m\geq 2$ by
  \begin{gather}
    \abs{\sum_{x_1<p<x_2}\sum_{m=2}^\infty \frac1m \left(\frac{\chi(p)}{p}\right)^m}  \leq \sum_{x_1<p<x_2}\frac{1}{p^2} \leq \int_{x_1}^\infty\frac{\dd t}{t^2}=\frac1{x_1}. \label{moo3}
  \end{gather}
  since   $\abs{\sum_2^\infty \frac{x^m}m}\leq \frac12 \sum_2^\infty \abs{x}^m=\frac{\abs{x}^2}{2(1-\abs x)}\leq x^2$ for any $\abs x\leq \frac 12$.
  For $m=1$ we have
  \begin{gather}
    \abs{\sum_{x_1<p<x_2}\dfrac{\chi(p)}p} \leq
    \sum_{x_1<p<x_2}\dfrac{1}p <
    \sum_{p\leq x_2}\dfrac{1}p - \sum_{p\leq x_1}\dfrac{1}p, \label{moo4}
  \end{gather}
  where the first term on the right-hand side can be bounded using inequality (6.21) in \cite{schoenfeld}, which states that under RH,
  \begin{gather}
    \sum_{p\leq x_2}\dfrac1p < \log\log x_2+ B + \dfrac{3\log x_2+4}{8\pi\sqrt x_2} \label{moo5} 
  \end{gather}
  for any $x_2\geq 13.5$.

  Combining the inequalities \eqref{moo1}, \eqref{moo2}, \eqref{moo3}, \eqref{moo4}, \eqref{moo5} we obtain
  \begin{gather}
    \abs{\log L(1,\chi) - \nu(x_1,d)} \leq \eta(x_1,x_2,d),
  \end{gather}
  and the inequality \eqref{happrox} follows from taking exponentials and applying the class number formula $h(d)=\frac{\sqrt{\abs d}}{\pi} L(1,\chi)$ for $d<-8$.
\end{proof}

\subsection {Computer resources}

Our program comprises 1500 lines of \verb|C++| code. The total time for the computation was 4.5 CPU years (or 6.5 hours in real-time, running on 6400 cores simultaneously), requiring 1 TB of temporary memory storage.

We used the computer package PARI (cf. \cite{pari}) to compute the
groups $H(d)$ and we used the computer package
\verb|primesieve| \cite{walish} to iterate through primes.



\begin{thebibliography}{99}

\bibitem{bach}
Eric Bach.
\newblock Improved approximations for {E}uler products.
\newblock In {\em Number theory ({H}alifax, {NS}, 1994)}, volume~15 of {\em CMS
  Conf. Proc.}, pages 13--28. Amer. Math. Soc., Providence, RI, 1995.

\bibitem{bankspappalardishparlinski} W. Banks, F. Pappalardi and I. Shparlinski. On group structures realized by elliptic curves over arbitrary finite fields. \emph{Experiment. Math.} \textbf{21} (2012), no. 1, 11--25.

\bibitem{belabas-fast-cubic}
K.~Belabas.
\newblock A fast algorithm to compute cubic fields.
\newblock {\em Math. Comp.}, 66(219):1213--1237, 1997.

\bibitem{bhargava-shankar-tsimerman-davenport-heilbronn-secondary}
M.~Bhargava, A.~Shankar, and J.~Tsimerman.
\newblock On the {D}avenport-{H}eilbronn theorems and second order terms.
\newblock {\em Invent. Math.}, 193(2):439--499, 2013.


\bibitem{boydkisilevsky} D. Boyd and H. Kisilevsky.  On the exponent of the ideal class groups of complex quadratic fields. \emph{Proc. Amer. Math. Soc.} \textbf{31} (1972), 433--436.

\bibitem{chowla}  S. Chowla.  An extension of Heilbronn's class number theorem. \emph{Quarterly J. Math.} \textbf{5} (1934), 304--307.

\bibitem{claborn}
L.~Claborn.
\newblock Every abelian group is a class group.
\newblock {\em Pacific J. Math.}, 18:219--222, 1966.

\bibitem{cohenlenstra} H. Cohen and H.W. Lenstra.  Heuristics on class groups of number fields.  Lecture Notes in Mathematics, \textbf{1068}, Springer, 1984, pp. 33--62.

\bibitem{cornell} G. Cornell.  Abhyankar's Lemma and the class group.  In \emph{Number theory, Carbondale 1979} (Proc. Southern
Illinois Conf., Southern Illinois Univ., Carbondale, Ill., 1979),
82--88, Lecture Notes in Math., 751, Springer, Berlin, 1979.

\bibitem{davidsmith} C. David and E. Smith.  A Cohen-Lenstra phenomenon for elliptic curves.  \emph{J. London Math. Soc.} \textbf{89} (2014) 24--44.

\bibitem{ellenberg-venkatesh-class-group-torsion-bounds}
J.~S. Ellenberg and A.~Venkatesh.
\newblock Reflection principles and bounds for class group torsion.
\newblock {\em Int. Math. Res. Not. IMRN}, (1):Art. ID rnm002, 18, 2007.

\bibitem{flajolet-sedgewick-book}
P. Flajolet and R. Sedgewick.
\newblock Analytic combinatorics. 
\newblock Cambridge University Press, Cambridge, 2009


\bibitem{granvillesound} A. Granville and K. Soundararajan.  The distribution of values of $L(1,\chi_d)$.  \emph{Geom. and Funct. Anal.} \textbf{13} (2003), 992--1028.

\bibitem{HR} G.~Hardy and S.~Ramanujan.  Asymptotic formulae in combinatory analysis.
\emph{Proc. London Math. Soc.} \textbf{2} (1918) 75--115.

\bibitem{heathbrown} D.R. Heath-Brown.  Imaginary quadratic fields with class group exponent $5$.  \emph{Forum Math.} \textbf{20} (2008), 275--283.

\bibitem{helfgott-venkatesh-class-group-three-torsion}
H.~A. Helfgott and A.~Venkatesh.
\newblock Integral points on elliptic curves and 3-torsion in class groups.
\newblock {\em J. Amer. Math. Soc.}, 19(3):527--550 (electronic), 2006.


\bibitem{hillarrhea} C. Hillar and D. Rhea.  Automorphisms of finite abelian groups.  \emph{Amer. Math. Monthly} \textbf{114} (2007), 917--923.

\bibitem{data-Fh} 
S. Holmin, P. Kurlberg.
List of $\mathcal F(h)$ for all odd $h<10^6$. Available at \url{https://people.kth.se/~kurlberg/class_group_data/class_group_orders.txt}
  
\bibitem{data-FG} 
S. Holmin, P. Kurlberg.
List of $\mathcal F(G)$ for all noncyclic $p$-groups $G$ of odd order $<10^6$. Available at \url{https://people.kth.se/~kurlberg/class_group_data/noncyclic_class_groups.txt}
  
\bibitem{data-d-Hd} 
S. Holmin, P. Kurlberg.
List of $(d,H(d))$ for all fundamental discriminants $d<0$ such that $H(d)$ is a noncyclic $p$-group of odd order $<10^6$. Available at \url{https://people.kth.se/~kurlberg/class_group_data/discriminants_of_noncyclic_groups.txt}
  



\bibitem{lagariasodlyzko} J. C. Lagarias and A. M. Odlyzko.  Effective versions of the Chebotarev density theorem, in A. Frohlich (ed.) \emph{Algebraic Number Fields}, pp. 409--464, Academic Press, 1977.



\bibitem{soundbound}
Youness Lamzouri, Xiannan Li, and Kannan Soundararajan.
\newblock Conditional bounds for the least quadratic non-residue and related
  problems.
\newblock {\em Math. Comp.}, 84(295):2391--2412, 2015.


\bibitem{leedham-green-occuring-class-groups}
C.~R. Leedham-Green.
\newblock The class group of {D}edekind domains.
\newblock {\em Trans. Amer. Math. Soc.}, 163:493--500, 1972.



\bibitem{lengler2} J. Lengler.  The Cohen-Lenstra Heuristic:  Methodology and results.  \emph{J. Algebra} \textbf{323} (2010), 2960--2976.

\bibitem{lengler} J. Lengler.  The global Cohen-Lenstra Heuristic.  \emph{J. Algebra} \textbf{357} (2012), 247--269.

\bibitem{ozaki-p-groups}
M.~Ozaki.
\newblock Construction of maximal unramified {$p$}-extensions with prescribed
  {G}alois groups.
\newblock {\em Invent. Math.}, 183(3):649--680, 2011.

\bibitem{pari}
{The PARI~Group}, Bordeaux.
\newblock {\em {PARI/GP version {\tt 2.7.3}}}, 2015.
\newblock Available at
\url{http://pari.math.u-bordeaux.fr/pub/pari/unix/pari-2.7.3.tar.gz}. 

\bibitem{perret-occuring-class-groups}
M.~Perret.
\newblock On the ideal class group problem for global fields.
\newblock {\em J. Number Theory}, 77(1):27--35, 1999.

\bibitem{pierce-three-torsion-bound}
L.~B. Pierce.
\newblock A bound for the 3-part of class numbers of quadratic fields by means
  of the square sieve.
\newblock {\em Forum Math.}, 18(4):677--698, 2006.


\bibitem{ranum}  A. Ranum.  The group of classes of congruent matrices with application to the group of isomorphisms of any abelian group.  \emph{Trans. Amer. Math. Soc.} \textbf{8} (1907), 71--91.


\bibitem{roberts-cubic-fields-secondary-term}
D.~P. Roberts.
\newblock Density of cubic field discriminants.
\newblock {\em Math. Comp.}, 70(236):1699--1705 (electronic), 2001.





\bibitem{schoenfeld}
Lowell Schoenfeld.
\newblock Sharper bounds for the {C}hebyshev functions {$\theta (x)$} and
  {$\psi (x)$}. {II}.
\newblock {\em Math. Comp.}, 30(134):337--360, 1976.

\bibitem{Sel1} James A. Sellers.  Extending a recent result of Santos
  on partitions into odd parts.  \emph{Integers} \textbf{3:A4}, 5
  pp. (electronic), 2003.
  
\bibitem{Sel2} James A. Sellers.  Corrigendum to: ``Extending a recent
  result of Santos on partitions into odd parts''.  \emph{Integers}
 \textbf{4:A8}, 1 pp. (electronic), 2004.
 
 \bibitem{serre} J.-P. Serre.  Quelques applications du th\'{e}or\`{e}me de densit\'{e} de Chebotarev. \emph{Publ. Math. I. H. E. S.} \textbf{54} (1981), 123--201.


\bibitem{sound} K. Soundararajan.  The number of imaginary quadratic fields with a given class number.  \emph{Hardy-Ramanujan J.} \textbf{30} (2007), 13--18.

\bibitem{taniguchi-thorne-secondary-term-cubic-fields}
T.~Taniguchi and F.~Thorne.
\newblock Secondary terms in counting functions for cubic fields.
\newblock {\em Duke Math. J.}, \textbf{162} (2013), 2451--2508.


\bibitem{walish} K. Walisch.  \verb|primesieve|. Fast C/C++ prime
  number generator.  Available at \url{http://primesieve.org/}.


\bibitem{watkins} M. Watkins.  Class numbers of imaginary quadratic
  fields.  \emph{Math. Comp.} \textbf{73} (2003), 907--938. 

\bibitem{weinberger} P. Weinberger.  Exponents of the class groups of complex quadratic fields.  \emph{Acta Arith.} \textbf{22} (1973), 117--124.

\bibitem{yahagi} O. Yahagi.  Construction of number fields with prescribed $l$-class groups.  \emph{Tokyo J. Math.} \textbf{1} (1978), no. 2, 275--€"283.

\bibitem{yamamoto}  Y. Yamamoto.  On unramified Galois extensions of quadratic number fields, \emph{Osaka J. Math.} \textbf{7} (1970), 57--76.

\end{thebibliography}

\end{document}